\documentclass[11pt,letterpaper]{amsart}

\usepackage{epsfig,amsmath,amsthm,amssymb,graphicx,hyperref}
\usepackage[top=1.25in, bottom=1.25in, left=1.25in, right=1.25in]{geometry}
\usepackage{caption}
\usepackage{rotating}

\newtheorem{proposition}{Proposition}[section]
\newtheorem{theorem}[proposition]{Theorem}

\newtheorem*{theorem*}{Theorem}

\newtheorem*{thm_4.3}{Theorem 4.3}
\newtheorem*{thm_3.1}{Theorem 3.1}

\newtheorem*{cor_4.5}{Corollary 4.5}
\newtheorem{corollary}[proposition]{Corollary}
\newtheorem{lemma}[proposition]{Lemma}
\theoremstyle{definition}
\newtheorem{definition}[proposition]{Definition}
\newtheorem{remark}[proposition]{Remark}

\DeclareMathOperator{\spn}{span}
\DeclareMathOperator{\lk}{lk}
\DeclareMathOperator{\rk}{rank}

\begin{document}
\title[Milnor's triple linking numbers and derivatives of genus three knots]{Milnor's triple linking number and derivatives of genus three knots}

\author{JungHwan Park}
\address{Department of Mathematics, Rice University MS-136\\
6100 Main St. P.O. Box 1892\\
Houston, TX 77251-1892}
\email{jp35@rice.edu}

\date{\today}


\begin{abstract}
A derivative of an algebraically slice knot $K$ is an oriented link disjointly embedded in a Seifert surface of $K$ such that its homology class forms a basis for a metabolizer $H$ of $K$. We show that for a genus three algebraically slice knot $K$, the set $\{ \bar{\mu}_{\{\gamma_1,\gamma_2,\gamma_3\}}(123) - \bar{\mu}_{\{\gamma'_1,\gamma'_2,\gamma'_3\}}(123)| \{\gamma_1,\gamma_2,\gamma_3\}$ and $\{\gamma'_1,\gamma'_2,\gamma'_3\}$ are derivatives of $K$ associated with a metabolizer $H\}$ contains $n\cdot \mathbb{Z}$ where $n$ is determined by a Seifert form of $K$ and a metabolizer $H$. As a corollary, we show that it is possible to realize any integer as a Milnor's triple linking number of a derivative of the unknot on a fixed Seifert surface with a fixed metabolizer. In addition, we show that a knot, which is a connected sum of three genus one algebraically slice knots, has at least one derivative which has non-zero Milnor's triple linking number.
\end{abstract}

\maketitle

\section{Introduction}\label{Introduction}
A knot $K$ is a smooth embedding of an oriented $S^1$ into $S^3$. If $K$ bounds a smoothly embedded $2$-disk $D^2$ in the standard $B^4$, we say $K$ is a smoothly slice knot. Recall that any knot in $S^3$ bounds a Seifert surface $F$ and from $F$ we have a Seifert form $\beta_F : H_1(F) \times H_1(F) \rightarrow \mathbb{Z}$. Here $\beta_F([x],[y]) = \lk(x,y^+)$, x is union of simple closed curves on $F$ that represents $[x]$, and $y^+$ is positive push off of union of simple closed curves on $F$ that represents $[y]$. If $K$ is a smoothly slice knot there exists a metaboliser $H \cong \mathbb{Z}^{\frac{1}{2} \rk H_1(F)}$, a direct summand of $H_1(F)$, such that $\beta_F$ vanishes on $H$ (see \cite[Lemma2]{Le69}). In this case we say that the Seifert form is metabolic, and we call a knot algebraically slice if it has a metabolic Seifert form. From now we will assume $K$ is algebraically slice and that $H$ is a metaboliser of $K$. Let $\{ \gamma_1, \gamma_2, \cdots, \gamma_{\frac{1}{2} \rk H_1(F)} \}$ be an oriented link disjointly embedded in the Seifert surface $F$ whose homology class forms a basis for $H$. Such a link is called a derivative of $K$ associated with $H$.

It is an interesting problem to understand the behavior of derivatives of a given algebraically slice knot. A lot of work has been done in this area (see \cite{Gi83,Li84,Gi93,GL92,GL92',GL13,CHL10,COT04,CD14,Park15}). However, not until recently has there been a study of the Milnor's invariants of derivatives (for the definition of Milnor invariants see section $2.2$). In \cite{JKP14}, Jang-Kim-Powell studied Milnor invariants of derivatives of a boundary links where each of the components has a Seifert surface of genus one. In this paper we will be focusing on a knot which has a genus three Seifert surface and trying to understand the behavior of its derivatives. In particular, we are interested in calculating their Milnor's triple linking number. It turns out that the behavior of Milnor's triple linking number of derivatives is quite complex as suggested by the following main theorem.

\begin{thm_4.3} Let $K$ be an algebraically slice knot with genus three Seifert surface $F$. Suppose $H$ is a metabolizer of $K$ and $\{ a_1, a_2, a_3, b_1, b_2, b_3 \}$ is a symplectic basis for $H_1(F)$ where $\spn(b_1,b_2,b_3) = H$. Let $$M =
 \begin{pmatrix}
  * & a & * & x_1 & * & y_1 \\
  a-1 & 0 & x_2 & 0 & y_2 & 0 \\
  * & x_2 & * & b & * & z_1 \\
  x_1 & 0 & b-1 & 0 & z_2 & 0 \\
  * & y_2 & * & z_2 & * & c \\
  y_1 & 0 & z_1 & 0 & c-1 & 0
 \end{pmatrix}$$ be the Seifert matrix with respect to the basis $\{ a_1, b_1, a_2, b_2, a_3, b_3 \}$. Then $S_{K,H} \supseteq \{n\cdot((a-1)(b-1)(c-1)-abc+x_1x_2+y_1y_2+z_1z_2)|n\in \mathbb{Z} \}$, where $S_{K,H} =\{ \bar{\mu}_{\{\gamma_1,\gamma_2,\gamma_3\}}(123) - \bar{\mu}_{\{\gamma'_1,\gamma'_2,\gamma'_3\}}(123)| \{\gamma_1,\gamma_2,\gamma_3\}$ and $\{\gamma'_1,\gamma'_2,\gamma'_3\}$ are derivatives of $K$ associated with $H\}$.
\end{thm_4.3}

It is still an open problem whether $S_{K,H}$ is in fact equal to $\{n\cdot((a-1)(b-1)(c-1)-abc+x_1x_2+y_1y_2+z_1z_2)|n\in \mathbb{Z} \}$. As a corollary of the main Theorem we show that even when we choose the knot to be the simplest possible knot, Milnor's triple linking number of derivatives can become very complicated.

\begin{cor_4.5} Let $U$ be the unknot with a Seifert surface $F$ given in Figure $1$. For $i=1,2,3$, let $\beta_i$ be the simple closed curve as in Figure $1$. If $H := \spn([\beta_1], [\beta_2], [\beta_3])$, then 
$$ \{ \bar{\mu}_{\{\gamma_1,\gamma_2,\gamma_3\}}(123) | \{\gamma_1,\gamma_2,\gamma_3\} \text{ is a derivative of U associated with H} \} = \mathbb{Z}$$
\end{cor_4.5}
\begin{figure}[h]
\centering
\includegraphics[width=6in]{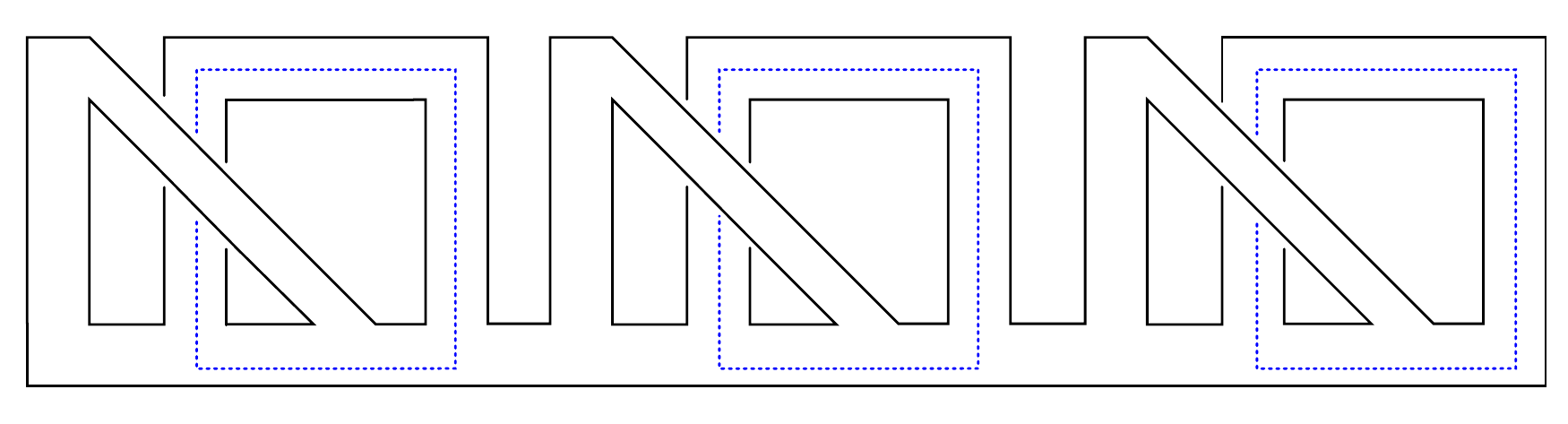}
\put(-4.8,-0){$\beta_1$}
\put(-2.8,-0){$\beta_2$}
\put(-0.8,-0){$\beta_3$}

\caption{Disk-band form of Seifert surface $F$ for the unknot $U$}
\end{figure}

In addition, we also show that a knot, which is a connected sum of three genus one algebraically slice knots, has at least one derivative which has non-zero Milnor's triple linking number.

\section{Preliminaries}\label{Preliminaries}

\subsection{Band form and Seifert matrix}
Fix a knot $K$ and its genus three Seifert surface $F$, then we have a Seifert form $\beta_F : H_1(F) \times H_1(F) \rightarrow \mathbb{Z}$. Since $K$ is an algebraically slice knot there exists a metaboliser $H=\mathbb{Z}^{\frac{1}{2} \rk H_1(F)} = \mathbb{Z}^3$, a direct summand of $H_1(F)$, such that $\beta_F$ vanishes on $H$. Let $\{ \gamma_1, \gamma_2, \cdots, \gamma_{\frac{1}{2} \rk H_1(F)} \} = \{ \gamma_1, \gamma_2, \gamma_3 \}$ be a derivative of $K$ associated with $H$. Let $b_1=[\gamma_1]$, $b_2=[\gamma_2]$, and $b_3=[\gamma_3]$ be elements in $H_1(F)$, then we can extend $\{b_1, b_2, b_3 \}$ to a symplectic basis $\{ a_1, a_2, a_3, b_1, b_2, b_3 \} $ where $a_i$ is an intersection dual of $b_i$ for each $i \in \{1,2,3\}$. This gives us a disk band form for the Seifert surface $F$ as in Figure $2$. From this we get a Seifert Matrix $$M =
 \begin{pmatrix}
  * & a & * & x_1 & * & y_1 \\
  a-1 & 0 & x_2 & 0 & y_2 & 0 \\
  * & x_2 & * & b & * & z_1 \\
  x_1 & 0 & b-1 & 0 & z_2 & 0 \\
  * & y_2 & * & z_2 & * & c \\
  y_1 & 0 & z_1 & 0 & c-1 & 0
 \end{pmatrix}$$ for $K$ with respect to a basis $\{ a_1, b_1, a_2, b_2, a_3, b_3 \}$, where $a,b,c,x_1,x_2,y_1,y_2,z_1$ and $z_2$ are integers.
 
\begin{figure}[h]
\centering
\includegraphics[width=6in]{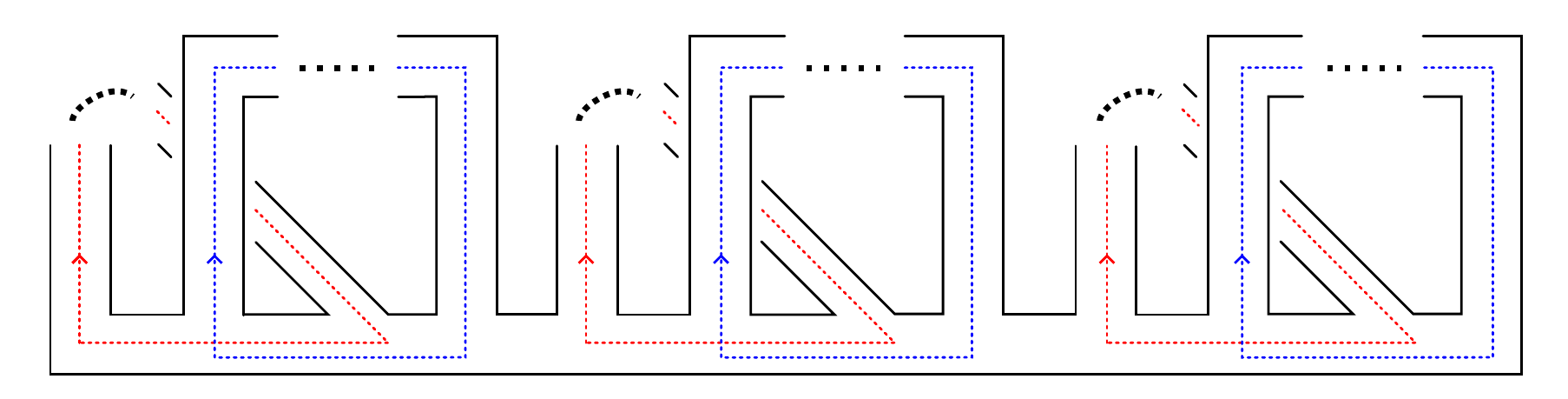}
\put(-5.5,-0.1){$a_1$}
\put(-4.8,-0.1){$b_1$}
\put(-3.5,-0.1){$a_2$}
\put(-2.8,-0.1){$b_2$}
\put(-1.5,-0.1){$a_3$}
\put(-0.8,-0.1){$b_3$}

\caption{Disk-band form of Seifert surface $F$ for $K$}
\end{figure}
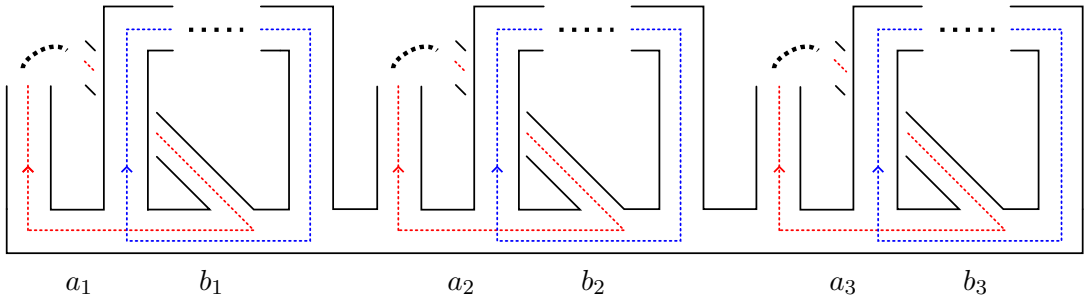

\subsection{Milnor's Triple Linking Number}
In this section, we will recall some definitions and properties of Milnor's triple linking number (for more general setting and detailed discussion see \cite{Mi54}, \cite{Mi57}, \cite{Or89}). We will focus on $3$-component link $L = \{ L_1, L_2, L_3 \}$ such that $\lk(L_i,L_j)=0$ for $i,j \in \{1,2,3 \}$ where $i \neq j$. Let $\mu_1, \mu_2$ and $\mu_3$ be meridians of $L_1, L_2$ and $L_3$ respectively and let $\lambda_1, \lambda_2$ and $\lambda_3$ be longitudes of $L_1, L_2$ and $L_3$ respectively. Let $\pi = \pi_1(S-L)$ and $F$ be a free group generated by $x_1, x_2$ and $x_3$. Also, when $G$ is a group, the $n$th lower central series of $G$ is denoted as $G_n$ and defined inductively by $G_1=G$ and $G_i=[G,G_{i-1}]$ for $i > 1$. Note that, since $\lk(L_i,L_j)=0$ for $i,j \in \{1,2,3 \}$ where $i \neq j$, $[\lambda_i] \in \pi_2$ for all $i \in \{1,2,3\}$. Hence, by \cite{Mi57} we have an isomorphism $F / F_3 \rightarrow \pi / \pi_3$ which is induced by $\overset{3}\vee S^1 \rightarrow S^3 - L$ where each of the circle from the wedge gets map to each of the meridian of the link $L$, $i.e.$ $x_1$ is identified with $[\mu_1]$, $x_2$ is identified with $[\mu_2]$, and $x_3$ is identified with $[\mu_3]$. So, from now on we will identify $F / F_3$ and $\pi / \pi_3$.

Recall Magnus representation $\phi : F\rightarrow P$ where $P$ is power series ring with non-commutative variable $a_1,a_2$ and $a_3$, and $\phi(x_i)=1+a_i$ and $\phi(x^{-1}_i)=1-a_i+a_i^2-a_i^3 + \cdots$ for $i \in \{1,2,3 \}$. Magnus proved that for any $y\in F$, $y \in F_k$ if and only if $\phi(y) = 1+ \sum c_i\cdot w_i$ where $c_i$ is an integer, $w_i$ is a word in $a_1, a_2$ and $a_3$, and length of $w_i$ is greater than equal to $k$ for all $i$ (see \cite{Ma04}).

We define Milnor's triple linking number $\bar{\mu}_{L}(123)$ of a link $L$ as the coefficient of $a_1\cdot a_2$ in $\phi([\lambda_3]) \in P$. Notice that by Magnus's work, $[\lambda_3] \in \pi_3$ if and only if $\bar{\mu}_{L}(123)=0$. It is also known that $F_2 / F_3 \simeq \pi_2 / \pi_3$ is a finitely generated free abelian group with a basis $\{[x_1,x_2], [x_1,x_3],[x_2, x_3]\}$ (see \cite{Ha50}). Then it is not hard to see that $\bar{\mu}_{L}(123)$ is equal to $n_1$ where $p([\lambda_3])=[x_1,x_2]^{n_1} \cdot [x_1,x_3]^{n_2} \cdot [x_2, x_3]^{n_{3}} \in F_2/F_3 = \pi_2 / \pi_3$ where $p$ is a projection $p : \pi_2 \rightarrow \pi_2/\pi_3$.

We will recall some properties of lower central series, which helps to compute $\bar{\mu}_{L}(123)$.

\begin{proposition}[See \cite{Ma04}]  Let $a,b,c$ be elements of a group $G$. Then
\begin{enumerate}
\item $[G_n,G_m] \subseteq G_{n+m}$ for any positive integers $n$ and $m$
\item $[a,b\cdot c] = [a,b]\cdot [a,c] \pmod{G_3}$
\item $[a\cdot b,c] = [a,c]\cdot [b,c] \pmod{G_3}$
\end{enumerate}
\end{proposition}

From Proposition $2.1(1)$ we have the following proposition.

\begin{proposition}[See \cite{Co90}] Let $\gamma$ be a loop in $S^3 - L$ and let $\gamma_1$ and $\gamma_2$ be basings of $\gamma$. If either  $[\gamma_1]$ or $[\gamma_2]$ is in $\pi_2$, then $[\gamma_1]=[\gamma_2]$ in $\pi_2 / \pi_3 = F_2 / F_3$.
\end{proposition}
\begin{proof} For some $g \in \pi$ we have $\gamma_1 = g \cdot \gamma_2 \cdot g^{-1} = g \cdot \gamma_2 \cdot g^{-1} \cdot \gamma_2^{-1} \cdot \gamma_2 = [g,\gamma_2] \cdot \gamma_2$. We can conclude the statement since $[g,\gamma_2]\in \pi_3$  by Proposition $2.1(1)$.
\end{proof}

This tells us that as long as any basing of a loop $\gamma$ in $S^3 - L$ is contained in $\pi_2$ we do not need to specify the basing. Using Proposition $2.1(2),(3)$ we have the following proposition immediately.

\begin{proposition} Let $w_1, w_2$ be elements of $F$ free group generated by $x_1,x_2$ and $x_3$. For $i\in \{ 1,2,3 \}$, let $n_i$ denote the exponent sum of $x_i$'s occurring in $w_1$, and let $m_i$ denote the exponent sum of $x_i$'s occurring in $w_2$. Then $[w_1,w_2] = [x_1, x_2]^{(n_1m_2-n_2m_1)} \cdot [x_1,x_3]^{(n_1m_3-n_3m_1)} \cdot [x_2,x_3]^{(n_2m_3-n_3m_2)}$ in $F_2 / F_3$.
\end{proposition}
\begin{proof}
By Proposition $2.1(2),(3)$ we can expand out $[w_1,w_2]$. Then the statement follows immediately.
\end{proof}

There is a nice geometric interpretation of Milnor' triple linking number of a link $L$ introduced by Cochran in \cite{Co90}. Given $L$ as above, suppose $F_1, F_2$ and $F_3$ are Seifert surfaces for $L_1, L_2$ and $L_3$ respectively. Further, for $i, j\in \{ 1,2,3 \}$ assume $F_i \cap L_j = \phi$ when $i\neq j$ and adjust so that triple intersection points are isolated. Then Milnor's triple linking number $\bar{\mu}_{L}(123)$ is the number of triple intersection points of $F_1, F_2$ and $F_3$, counted with signs. The sign of intersection $p\in F_1 \cap F_2 \cap F_3$ is positive if and only if $\{ \overrightarrow{v_1}, \overrightarrow{v_2}, \overrightarrow{v_3} \}$ agrees with standard orientation of $S^3$ where $\overrightarrow{v_i}$  is a normal vector to $F_i$ at p, for $i\in \{ 1,2,3 \}$. This tells us that $\bar{\mu}_{L}(\sigma(123)) = sign(\sigma)\cdot \bar{\mu}_{L}(123)$ for $\sigma \in S_3$, and change of an orientation of a component of $L$ changes the sign of Milnor' triple linking number.

\subsection{Infection by String Links}
In this section we will recall definitions of string link and string link infection. The following definitions are from \cite{JKP14}.

\begin{definition} (String Link)
\begin{enumerate}

\item An $r$-multi-disk $\mathbb{E}$ is an oriented disk $D^2$ with $r$ ordered embedded open disks $E_1, E_2, \cdots, E_r$. We have pairwise disjoint paths $\gamma_1, \gamma_2, \cdots, \gamma_r$ such that $\gamma_i(0) \in \partial E_i$ and $\gamma_i(1) \in \partial \mathbb{E}$ (see Figure $3$).

\begin{figure}[h]
\centering
\includegraphics[width=1.8in]{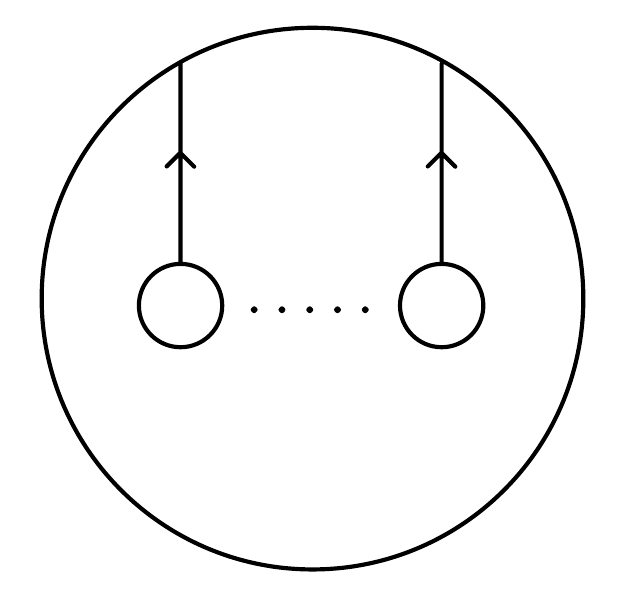}
\put(-0.9,-0.1){$\mathbb{E}$}
\put(-1.35,0.58){$E_1$}
\put(-0.61,0.58){$E_r$}
\put(-1.2,1.2){$\gamma_1$}
\put(-0.46,1.2){$\gamma_r$}

\caption{An $r$-multi-disk $\mathbb{E}$}
\end{figure}
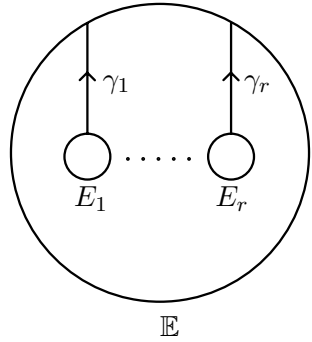

\item Let $p_i$ be the fixed point in the interior of $E_i$ for each $i=1,2, \cdots, r$. An $r$-component string link is a smooth proper embedding $J : \bigsqcup\limits_{i=1}^r p_i \times I \rightarrow D^2 \times I$ such that $J(p_i \times \{ 0 \} ) = p_i \times \{ 0 \} \in D^2 \times \{ 0 \}$ and $J(p_i \times \{ 1 \} ) = p_i \times \{ 1 \} \in D^2 \times \{ 1 \}$ for each $i = 1, 2, \cdots, r$. We will denote $\nu(J) : \bigsqcup\limits_{i=1}^r p_i \times I \rightarrow D^2 \times I$ as a tubular neighborhood of $J$. We will abuse notation and use $J$ and $\nu(J)$ as image of $J$ and image of $\nu(J)$ respectively.

\item For $i = 1, 2, \cdots, r$, the meridian of a $i$th component of $J$ is the simple closed curve, up to ambient isotopy, on the boundary of $\nu(J)(E_i \times I)$, which has the linking number $1$ with the $i$th component of $J$. For $i = 1, 2, \cdots, r$, let $\delta_i : I \rightarrow \partial E_i \times I$ be a $0$-framed parallel of $i$th component such that $\delta_i(0) = \gamma_i(0) \times \{ 1 \} \in \mathbb{E} \times I$ and $\delta_i(1) = \gamma_i (0) \times \{ 1 \} \in \mathbb{E} \times I$. Then the longitude $l_i$ of the $i$th component of $J$ is the concatenation of arcs as follows: $l_i = \delta_i \cup (\gamma_i \times \{ 1 \}) \cup (\gamma_i(1) \times I) \cup (-\gamma_i \times \{0\})$.
\end{enumerate}
\end{definition}

The following definitions are also from \cite{JKP14}.
\begin{definition} (Infection by a string link)
Let $L$ be a link in $S^3$ and $J$ be an $r$-component string link in $D^2\times I$.
\begin{enumerate}

\item Let $\mathbb{E}$ be an $r$-multi disk, then an embedding $\varphi : \mathbb{E} \rightarrow S^3$ is a proper $r$-multi disk in $(S^3,L)$ if $\varphi(\mathbb{E})$ intersects $L$ only in $\varphi(E_i)$ transversely for $i=1,2, \cdots, r$.

\item Let $\mathbb{E}_\varphi$ be the image of $\mathbb{E}$ under $\varphi$ and let $E_\varphi$ be the image of $E_1 \bigsqcup E_2 \bigsqcup \cdots \bigsqcup E_r$ under $\varphi$. We define the link $S(L,J,\varphi)$ to be the image of the link $L$ under the following homeomorphism (for detailed discussion for this homeomorphism see \cite{CFT09}):
\begin{equation*}
\begin{split}
   & (S^3 \setminus (int(\mathbb{E}_\varphi \setminus E_\varphi) \times I)) \cup ((D^2 \times I) \setminus \nu(J))\\
= & (S^3 \setminus (\mathbb{E}_\varphi \times I)) \cup (((D^2 \times I) \setminus \nu(J)) \cup (\overline{\mathbb{E}_\varphi} \times I))\\
\cong  & S^3.
\end{split}
\end{equation*}
\end{enumerate}
\end{definition}

We call the resulting link $S(L,J,\varphi)$ multi-infection of $L$ by $J$ along $\mathbb{E}_\varphi$. In Figure $4$ we present an example of $S(L,J,\varphi)$ for a particular $L, J$ and $\mathbb{E}_\varphi$.

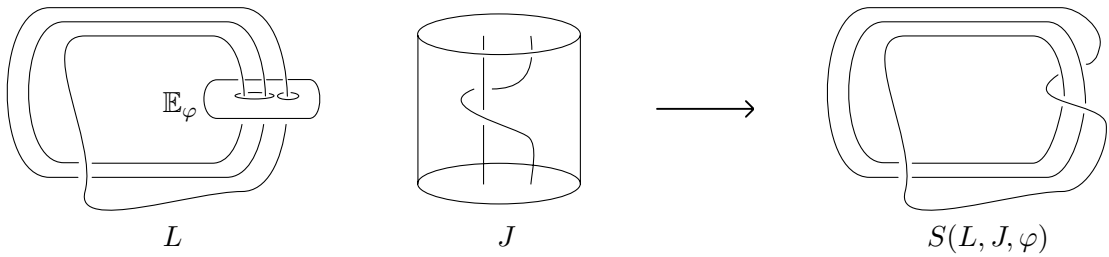
\begin{figure}[h]
\centering
\includegraphics[width=6in]{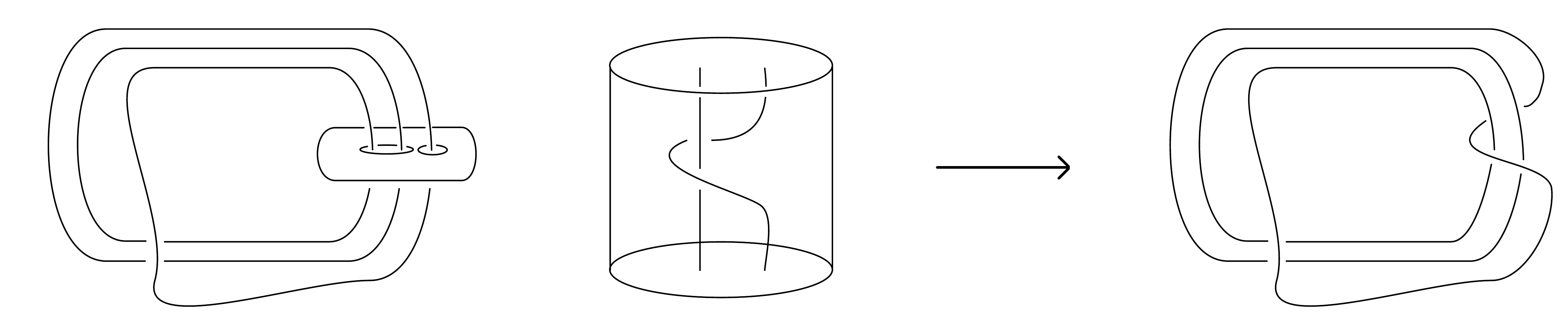}
\put(-1,-0.1){$S(L,J,\varphi)$}
\put(-5,-0.1){$L$}
\put(-3.25,-0.1){$J$}
\put(-5,0.6){$\mathbb{E}_\varphi$}

\caption{Infection by a string link}
\end{figure}

\subsection{Geometric Moves on Knots}
In this section we will recall definitions of double delta move and double Borromean rings insertion move and see how they are related. The following definitions are from \cite{Ma15}.

\begin{definition} (Double delta move)
A double delta move on a knot $K$ is the local move described in Figure $5$.
\end{definition}

\begin{figure}[h]
\centering
\includegraphics[width=4in]{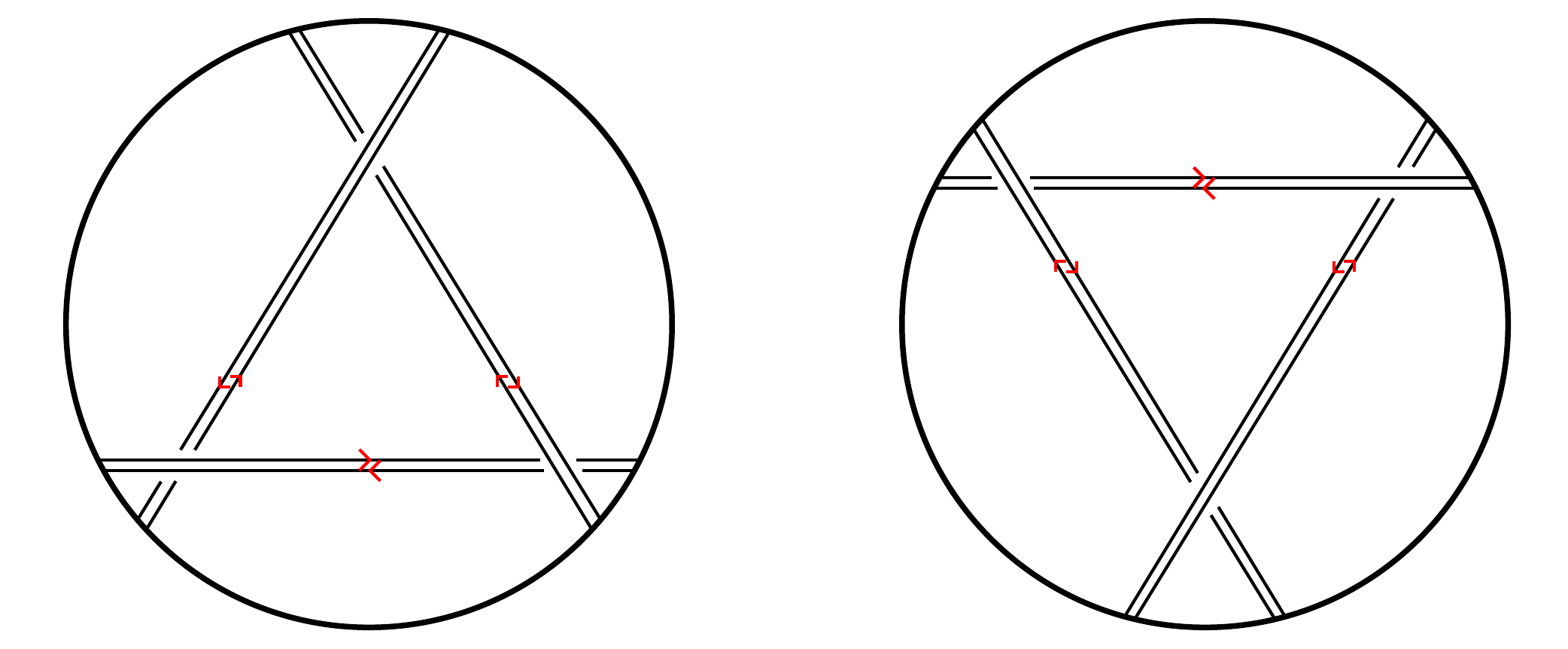}
\put(-2.05,0.85){$\Leftrightarrow$}

\caption{A double delta move}
\end{figure}

\begin{definition} (Double Borromean rings insertion move)
A double Borromean rings insertion move on a knot $K$ is the local move described in Figure $6$.
\end{definition}

\begin{figure}[h]
\centering
\includegraphics[width=4in]{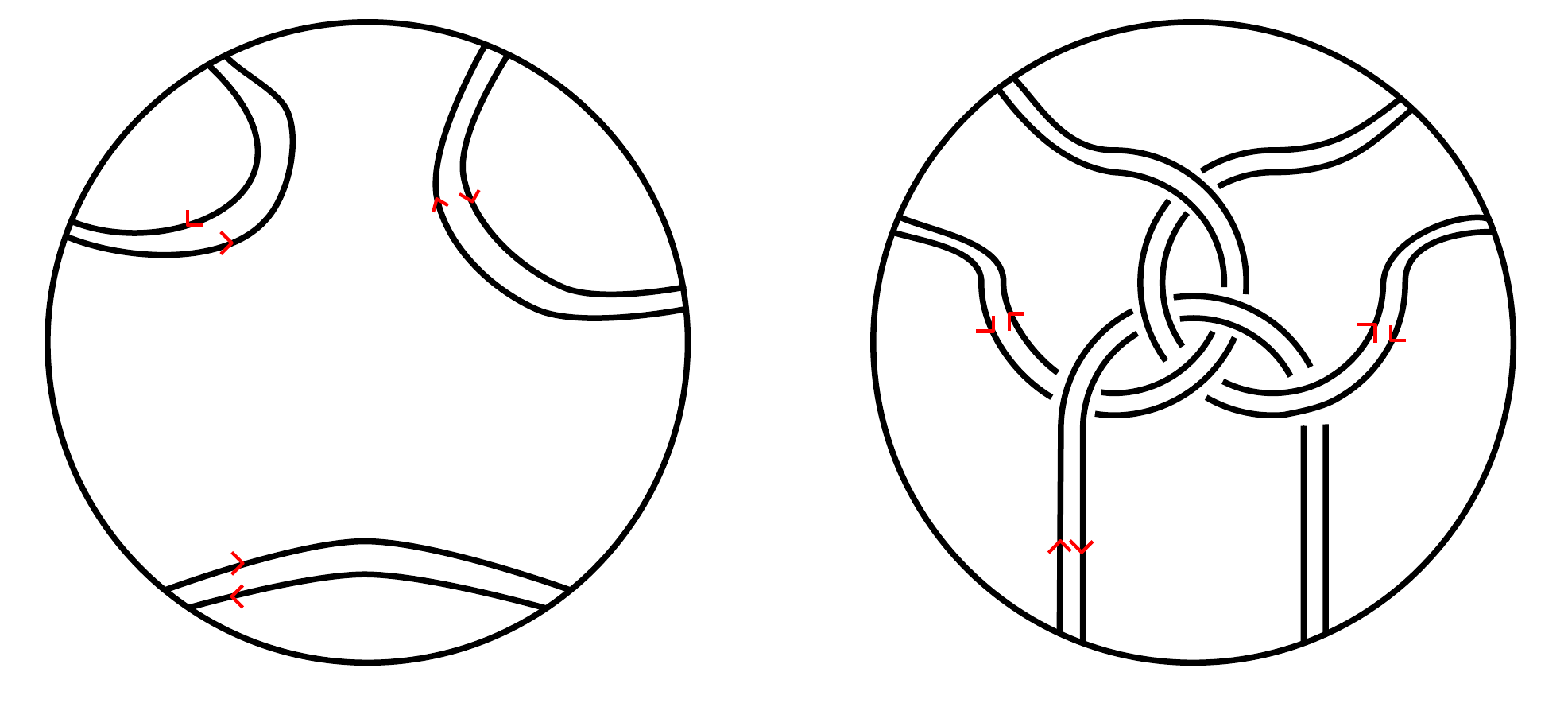}
\put(-2.05,0.85){$\Rightarrow$}

\caption{A double Borromean rings insertion move}
\end{figure}

\begin{remark} $ $
\begin{enumerate}
\item Notice that a double Borromean rings insertion move is a special case of a string link infection $S(L,J,\varphi)$ where $\varphi : \mathbb{E} \rightarrow S^3$ is given as in Figure $7$, and $\widehat{J}$ is a Borromean ring.
\item A double delta move can be achieved by a double Borromean rings insertion move which is explained in Figure $8$.
\end{enumerate}
\end{remark}

\begin{figure}[h]
\centering
\includegraphics[width=3in]{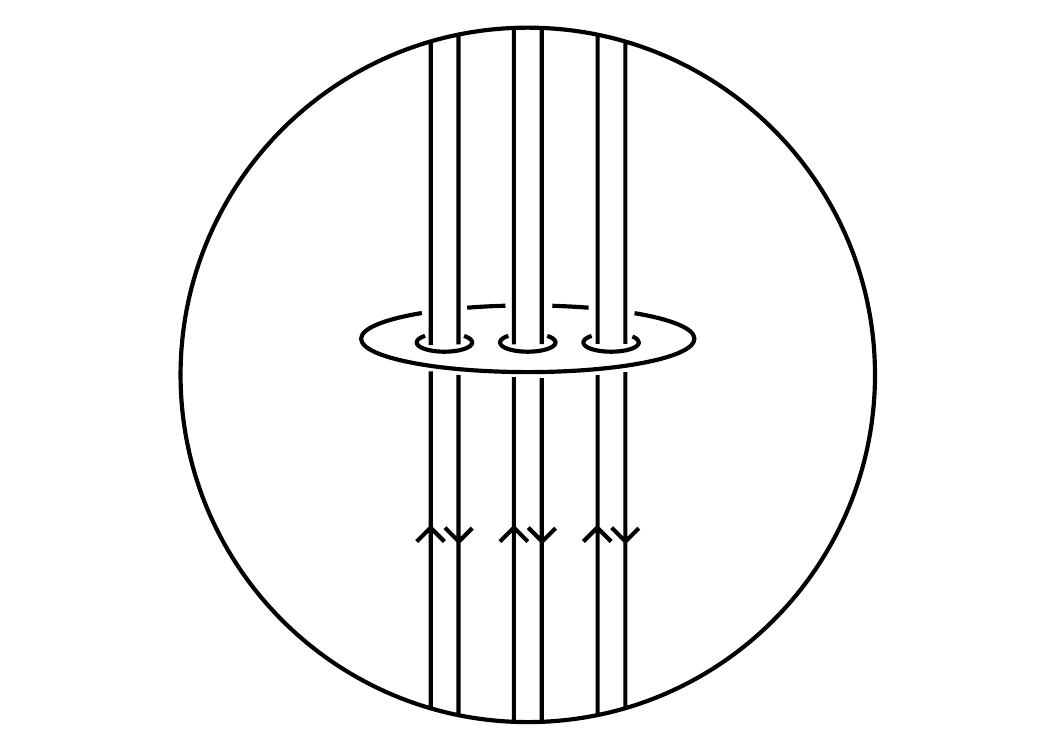}
\put(-2,1.35){$\mathbb{E}_\varphi$}
\put(-1.1,0.7){$L$}

\caption{A double Borromean rings insertion move is a special case of a string link infection.}
\end{figure}
\begin{figure}[h]
\centering
\includegraphics[width=6in]{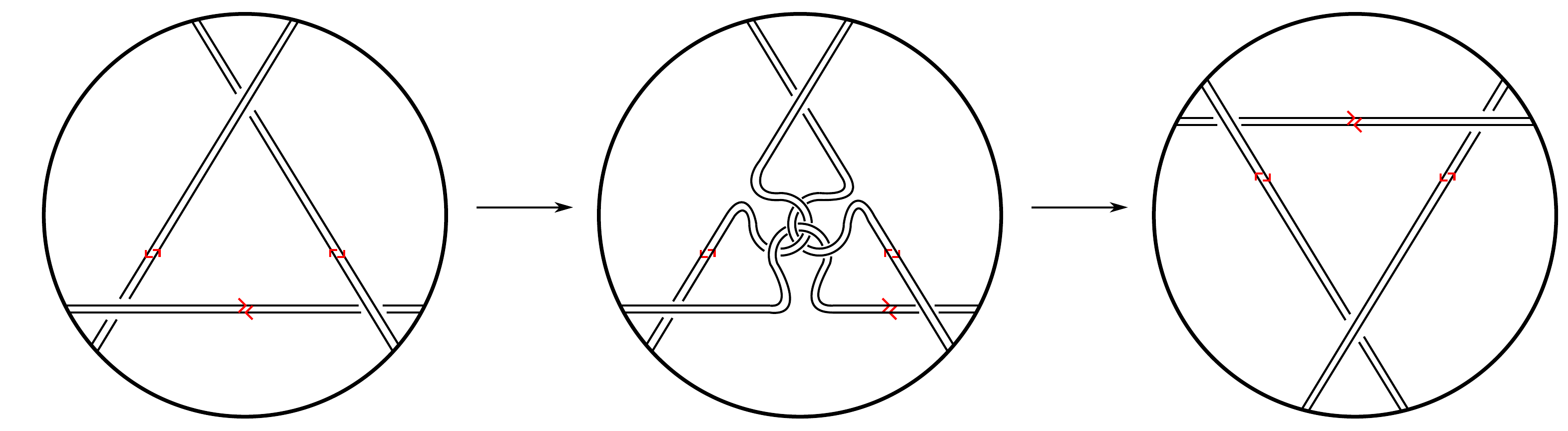}
\put(-2.1,-0.1){isotopy}
\put(-4.55,-0.1){double Borromean}
\put(-4.55,-0.25){rings insertion move}

\caption{A double delta move achieved by a double Borromean rings insertion move.}
\end{figure}

Let $K_1$ and $K_2$ be knots with a Seifert surface $F_1$ and $F_2$ respectively. Suppose that genus of $F_1$ and $F_2$ are the same and the Seifert Form of $K_1$ and $K_2$ are the same. Then it is known that it is possible to alter $K_1$ only by applying double delta move on the band of $F_1$ to obtain $K_2$ (see \cite{NS03}). Then by the Remark $2.8. (2)$ we can get to $K_2$ from $K_1$ only by applying double Borromean rings insertion moves on the band of $F_1$. We will use this fact later in Section $4$.

\section{The effect of string link infection on Milnor's triple linking number}
In this section we will recall a Lemma from \cite[Lemma 4.1]{JKP14}. We will change assumptions slightly from \cite[4.1]{JKP14} for the purpose of this paper. For completeness we will also present the proof of the Lemma from \cite[Lemma 4.1]{JKP14}.

\begin{lemma} \cite[Lemma 4.1]{JKP14} Let $L= L_1 \bigsqcup L_2 \bigsqcup L_3$ be an oriented three component link with pairwise linking number zero, and let $J$ be an oriented three component string link whose closure $\widehat{J}$ has pairwise linking number zero. Let $\varphi : \mathbb{E} \rightarrow S^3$ be a proper 3-multi disc in $(S^3,L)$. Denote the algebraic intersection number between $\varphi(E_i)$ and $L_j$ by $n^j_i$ for $i,j \in \{1,2,3\}$. Then $S(L,J,\varphi)$ has zero pairwise linking number and $\bar{\mu}_{S(L,J,\varphi)}(123) = \bar{\mu}_{\widehat{J}}(123)(\sum\limits_{\sigma \in S_3} sign(\sigma)n^{\sigma(1)}_1n^{\sigma(2)}_2n^{\sigma(3)}_3) + \bar{\mu}_{L}(123)$.
\end{lemma}
\begin{proof}

It is straightforward to see that $S(L,J,\varphi)$ has pairwise linking number zero, since $\widehat{J}$ has pairwise linking number zero. Asumme that $\varphi(E_i)$ intersects the link $L$ transversely. Let $\alpha^j_i$ be the number of positive intersection and $\beta^j_i$ be the number of negative intersection between $\varphi(E_i)$ and $L_j$ for $i,j \in \{1,2,3\}$. Let $J'$ be the oriented string link where we take $\alpha^1_i+\alpha^2_i+\alpha^3_i$ many parallel copy of $i$th component with the same orientation and $\beta^1_i+\beta^2_i+\beta^3_i$ many parallel copy of $i$th component with the opposite orientation for $i\in \{1,2,3\}$. Then we can consider $S(L,J,\varphi)$ as the result of performing several exterior band sums between $L$ and $\widehat{J'}$, the closure of $J'$ (see Figure $9$).

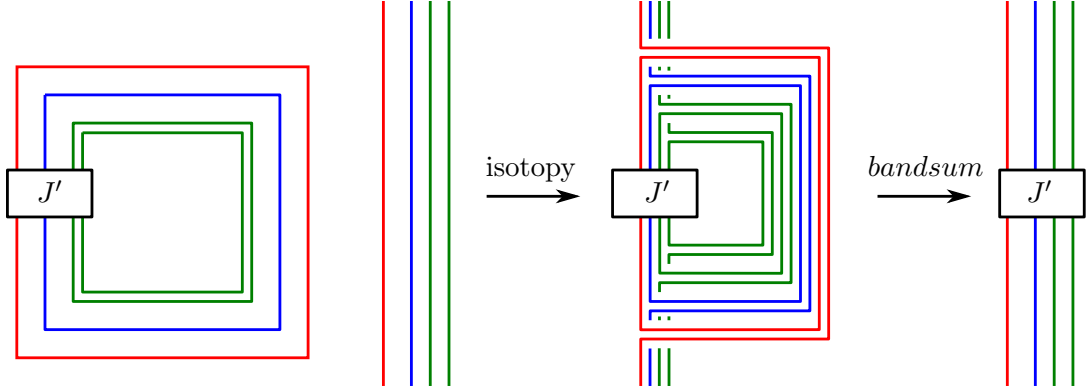
\begin{figure}[h]
\centering
\includegraphics[width=6in]{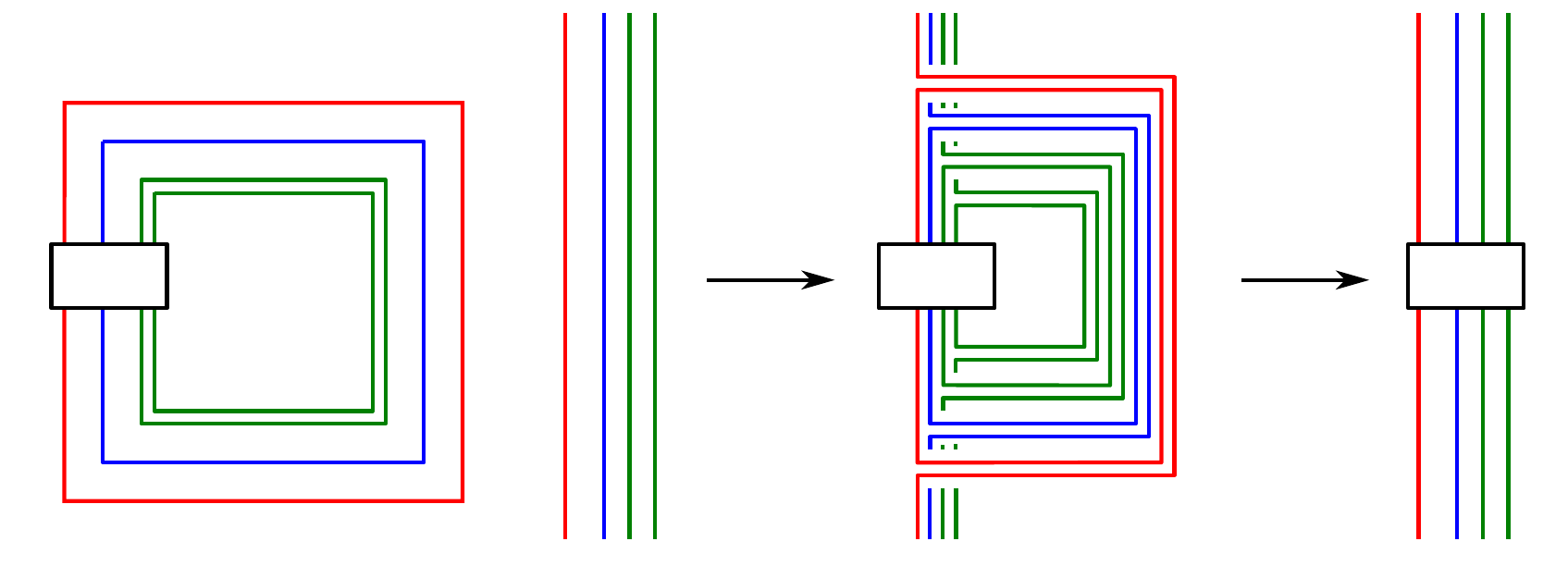}
\put(-5.65,1.06){$J'$}
\put(-2.47,1.06){$J'$}
\put(-0.47,1.06){$J'$}
\put(-3.3,1.2){isotopy}
\put(-1.3,1.2){$band sum$}

\caption{Band sums between $L$ and $\widehat{J'}$, the closure of $J'$.}
\end{figure}

We will label each component of $J'$ with index $\{1,2, \cdots, \sum\limits_{i,j} (\alpha^j_i + \beta^j_i)\}$ and we define a set map $h : \{1,2, \cdots, \sum\limits_{i,j} (\alpha^j_i + \beta^j_i)\} \rightarrow \{1,2,3\} $ where, for $n\in \{1,2, \cdots, \sum\limits_{i,j} (\alpha^j_i + \beta^j_i)\}$, $h(n)$ is the index of the component of the link $L=L_1\bigsqcup L_2 \bigsqcup L_3 $ where $n$th component of $J'$ is being banned summed to. It is known that the first non-vanishing Milnor invariant is additive under the exterior band sum (see \cite{Co90}). Hence, we have the following equation : 
\begin{equation*}
\bar{\mu}_{S(L,J,\varphi)}(123) = \sum\limits_{\{ I' \subset \{1,2, \cdots, \sum\limits_{i,j} (\alpha^j_i + \beta^j_i)\} | h(I')=\{1,2,3\}, |I'|=3 \}} sign(h(I'_1)h(I'_2)h(I'_3)) \cdot \bar{\mu}_{\widehat{J}}(123) +  \bar{\mu}_L(123) 
\end{equation*}
where $I'_i$ is $i$th component of $I'$ for $i=1,2,3$ and we are considering $(h(I'_1)h(I'_2)h(I'_3))$ as an element of $S_3$.

Recall that reversing the orientation of a component of a link changes the sign of the Milnor's triple linking number (see section 1.2). Also, note that $n^i_j$, the algebraic intersection number between $\varphi(E_i)$ and $L_j$, is equal to $\alpha^j_i - \beta^j_i$. Combining these with the above equation, we have the following desired equation : 

\begin{equation*}
\begin{split}
\bar{\mu}_{S(L,J,\varphi)}(123) = \: & \bar{\mu}_{\widehat{J}}(123)\cdot\sum\limits_{\sigma \in S_3} (sign(\sigma)\cdot(\alpha^{\sigma(1)}_1\alpha^{\sigma(2)}_2\alpha^{\sigma(3)}_3-\beta^{\sigma(1)}_1\alpha^{\sigma(2)}_2\alpha^{\sigma(3)}_3-\alpha^{\sigma(1)}_1\beta^{\sigma(2)}_2\alpha^{\sigma(3)}_3-\\
&\alpha^{\sigma(1)}_1\alpha^{\sigma(2)}_2\beta^{\sigma(3)}_3+\beta^{\sigma(1)}_1\beta^{\sigma(2)}_2\alpha^{\sigma(3)}_3+\alpha^{\sigma(1)}_1\beta^{\sigma(2)}_2\beta^{\sigma(3)}_3+\beta^{\sigma(1)}_1\alpha^{\sigma(2)}_2\beta^{\sigma(3)}_3-\\
&\beta^{\sigma(1)}_1\beta^{\sigma(2)}_2\beta^{\sigma(3)}_3)) +  \bar{\mu}_L(123)  \\
= \: & \bar{\mu}_{\widehat{J}}(123)\cdot\sum\limits_{\sigma \in S_3} (sign(\sigma)\cdot(n^{\sigma(1)}_1\alpha^{\sigma(2)}_2\alpha^{\sigma(3)}_3-\alpha^{\sigma(1)}_1\beta^{\sigma(2)}_2n^{\sigma(3)}_3-n^{\sigma(1)}_1\alpha^{\sigma(2)}_2\beta^{\sigma(3)}_3+\\
& \beta^{\sigma(1)}_1\beta^{\sigma(2)}_2n^{\sigma(3)}_3)) + \bar{\mu}_L(123) \\
= \: & \bar{\mu}_{\widehat{J}}(123)\cdot\sum\limits_{\sigma \in S_3} (sign(\sigma)\cdot(n^{\sigma(1)}_1\alpha^{\sigma(2)}_2n^{\sigma(3)}_3-n^{\sigma(1)}_1\beta^{\sigma(2)}_2n^{\sigma(3)}_3)) + \bar{\mu}_L(123) \\
= \: & \bar{\mu}_{\widehat{J}}(123)\cdot\sum\limits_{\sigma \in S_3} (sign(\sigma)\cdot(n^{\sigma(1)}_1n^{\sigma(2)}_2n^{\sigma(3)}_3)) + \bar{\mu}_L(123).
\end{split}
\end{equation*}

\end{proof}
\section{Main theorem and applications}\label{Main Theorem and applications}
In this section we will present the main theorem and see its applications. Before that we prove two lemmas which will be useful. For any algebraically slice knot $K$ with metabolizer $H$, where $\{ b_1, b_2, b_3 \}$ is a basis for $H$, we will define $S_{K,H,\{b_1,b_2,b_3\}} $ as a set of integers $\{ \bar{\mu}_{\{\gamma_1,\gamma_2,\gamma_3\}}(123) - \bar{\mu}_{\{\gamma'_1,\gamma'_2,\gamma'_3\}}(123)| \{\gamma_1,\gamma_2,\gamma_3\}$ and $\{\gamma'_1,\gamma'_2,\gamma'_3\}$ are derivatives of $K$ associated with $H$, $b_1 = [\gamma_1] = [\gamma'_1], b_2=[\gamma_2]=[\gamma'_2], b_3=[\gamma_3]=[\gamma'_3] \}$. We will also define $S_{K,H} $ as a set of integers $\{ \bar{\mu}_{\{\gamma_1,\gamma_2,\gamma_3\}}(123) - \bar{\mu}_{\{\gamma'_1,\gamma'_2,\gamma'_3\}}(123)| \{\gamma_1,\gamma_2,\gamma_3\}$ and $\{\gamma'_1,\gamma'_2,\gamma'_3\}$ are derivatives of $K$ associated with $H\}$.

\begin{lemma} Let $K$ and $\widetilde{K}$ be algebraically slice knots, and suppose $F$ and $\widetilde{F}$ are genus three Seifert surfaces for $K$ and $\widetilde{K}$ respectively. Let $\{ a_1, a_2, a_3, b_1, b_2, b_3\} $ be a symplectic basis for $H_1(F)$ where $\spn(b_1,b_2,b_3)$ is a metaboliser $H$ for $K$, and let  $\{ \widetilde{a_1}, \widetilde{a_2}, \widetilde{a_3}, \widetilde{b_1}, \widetilde{b_2}, \widetilde{b_3}\} $ be a symplectic basis for $H_1(\widetilde{F})$ where $\spn(\widetilde{b_1},\widetilde{b_2},\widetilde{b_3})$ is a metaboliser $\widetilde{H}$ for $\widetilde{K}$. We will denote $M$ as the Seifert matrix which arises from the basis $\{ a_1, b_1, a_2, b_2, a_3, b_3\} $ and $\widetilde{M}$ as the Seifert matrix which arises from the basis $\{ \widetilde{a_1}, \widetilde{b_1}, \widetilde{a_2}, \widetilde{b_2}, \widetilde{a_3}, \widetilde{b_3}\} $. If $M=\widetilde{M}$, we have $S_{K,H,\{b_1,b_2,b_3\}} = S_{\widetilde{K},\widetilde{H},\{\widetilde{b_1},\widetilde{b_2},\widetilde{b_3}\}}$ as a set.
\end{lemma}

\begin{proof}
Using symplectic bases $\{ a_1, a_2, a_3, b_1, b_2, b_3\} $ and $\{ \widetilde{a_1}, \widetilde{a_2}, \widetilde{a_3}, \widetilde{b_1}, \widetilde{b_2}, \widetilde{b_3}\} $ we can get disk-band form for $F$ and $\widetilde{F}$. As we previously observed at the end of section $2$, we can get to $\widetilde{K}$ from $K$ only by performing double Borromean rings insertion moves on the band of $F$, since $K$ and $\widetilde{K}$ have the same Seifert form. Also, since we can think of double Borromean rings insertion move as a special case of a string link infection (see Remark $2.8.(1)$), we can apply Lemma $3.1$ to our situation. Suppose $L = \{\gamma_1,\gamma_2,\gamma_3\}$ and $L'=\{\gamma'_1,\gamma'_2,\gamma'_3\}$ are derivatives of $K$ associated with $H$, where $[\gamma_1]=[\gamma'_1]=b_1, [\gamma_2]=[\gamma'_2]=b_2$ and $[\gamma_3]=[\gamma'_3]=b_3$. If we perform one double Borromean rings insertion move on bands of $F$, we get two links $S(L,J,\varphi)$ and $S(L',J,\varphi)$, where $\widehat{J}$ is a Borromean rings. By Lemma $3.1$, for $i,j \in \{1,2,3\}$, Milnor's triple linking number of $S(L,J,\varphi)$ only depends on the algebraic intersection number between $\varphi(E_i)$ and $\gamma_j$, and Milnor's triple linking number of $S(L',J,\varphi)$ only depends on the algebraic intersection number between $\varphi(E_i)$ and $\gamma'_j$. Since $[\gamma_j]=[\gamma'_j]\in H_1(F)$ for $j=1,2,3$, we know that they have the same algebraic intersection number with $\varphi(E_i)$ for $i=1,2,3$. Therefore, we can conclude $\bar\mu_{L}(123)-\bar\mu_{L'}(123) = \bar\mu_{S(L,J,\varphi)}(123)-\bar\mu_{S(L',J,\varphi)}(123)$, i.e. double Borromean rings insertion move does not change the difference of Milnor's triple linking number of two derivatives. By applying more double Borromean rings insertion moves, we can achieive $\widetilde{K}$. Using the same argument, we know that difference of their Milnor's triple linking number after all the double Borromean rings insertion moves is still $\bar\mu_{L}(123)-\bar\mu_{L'}(123)$. Therefore we can conclude that $S_{K,H,\{b_1,b_2,b_3\}} \subseteq S_{\widetilde{K},\widetilde{H},\{\widetilde{b_1},\widetilde{b_2},\widetilde{b_3}\}}$. We get the other inclusion by simply switching the roles of $K$ and $\widetilde{K}$.
\end{proof}

We will present one more lemma which will be useful for the main theorem.

\begin{lemma} Let $K$ be an algebraically slice knot and let $F$ be a genus three Seifert surface for $K$. Suppose $H$ is a metaboliser of $K$ and $\{ a_1, a_2, a_3, b_1, b_2, b_3 \}$ is a symplectic basis for $H_1(F)$ where $\spn(b_1,b_2,b_3) = H$, then $S_{K,H} = S_{K,H,\{ b_1, b_2, b_3 \} } \cup -S_{K,H,\{ b_1, b_2, b_3 \} } $. (i.e. Choice of a basis does not matter once we pick a metaboliser for $K$.)
\end{lemma}

\begin{proof}
It is obvious to see $S_{K,H} \supseteq S_{K,H,\{ b_1, b_2, b_3 \} } \cup -S_{K,H,\{ b_1, b_2, b_3 \} } $, since $-S_{K,H,\{ b_1, b_2, b_3 \} } = S_{K,H,\{ -b_1, b_2, b_3 \} } $.

For the other direction, we suppose $\{ \gamma_1, \gamma_2, \gamma_3 \}$ is a derivative of $K$ associated with $H$. Take a parallel copy of $\gamma_2$ and perform a band sum on the Seifert surface $F$ from $\gamma_1$ to the parallel copy of $\gamma_2$ (see Figure $10$). Let $\gamma'_1$ be the resulting knot after the band sum. Also, let $\gamma'_2=\gamma_2$ and $\gamma'_3=\gamma_3$. We claim that $\bar\mu_{\{\gamma_1, \gamma_2, \gamma_3 \}}(123) = \bar\mu_{\{\gamma'_1, \gamma'_2, \gamma'_3 \}}(123)$. Recall from section $2.2$ that there is a nice geometric interpretation of Milnor's triple linking number which was counting the number of triple intersection points of Seifert surfaces with sign. Let $F_1, F_2$ and $F_3$ be Seifert surfaces for $\gamma_1, \gamma_2$ and $\gamma_3$ respectively, where $F_i \cap \gamma_j = \phi$ for $i \neq j$ and $F_1, F_2$ and $F_3$ have isolated triple intersection points. We will take a parallel copy of $F_2$ which bounds the parallel copy of $\gamma_2$ and take a boundary sum with $F_1$ along the band which was used to perform band sum between $\gamma_1$ and the parallel copy of $\gamma_2$ (see Figure $10$). Let $F'_1$ be the resulting surface which bounds $\gamma'_1$. Also, let $F'_2=F_2$ and $F'_3=F_3$. Then notice that we have not introduced any new triple intersection points between $F'_1, F'_2$ and $F'_3$, since parallel copy of $F_2$ and $F_2$ does not intersect. 

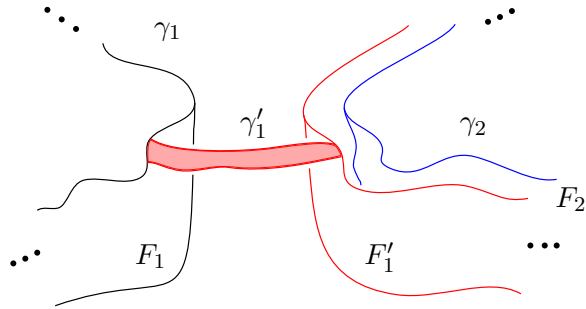
\begin{figure}[h]
\centering
\includegraphics[width=3in]{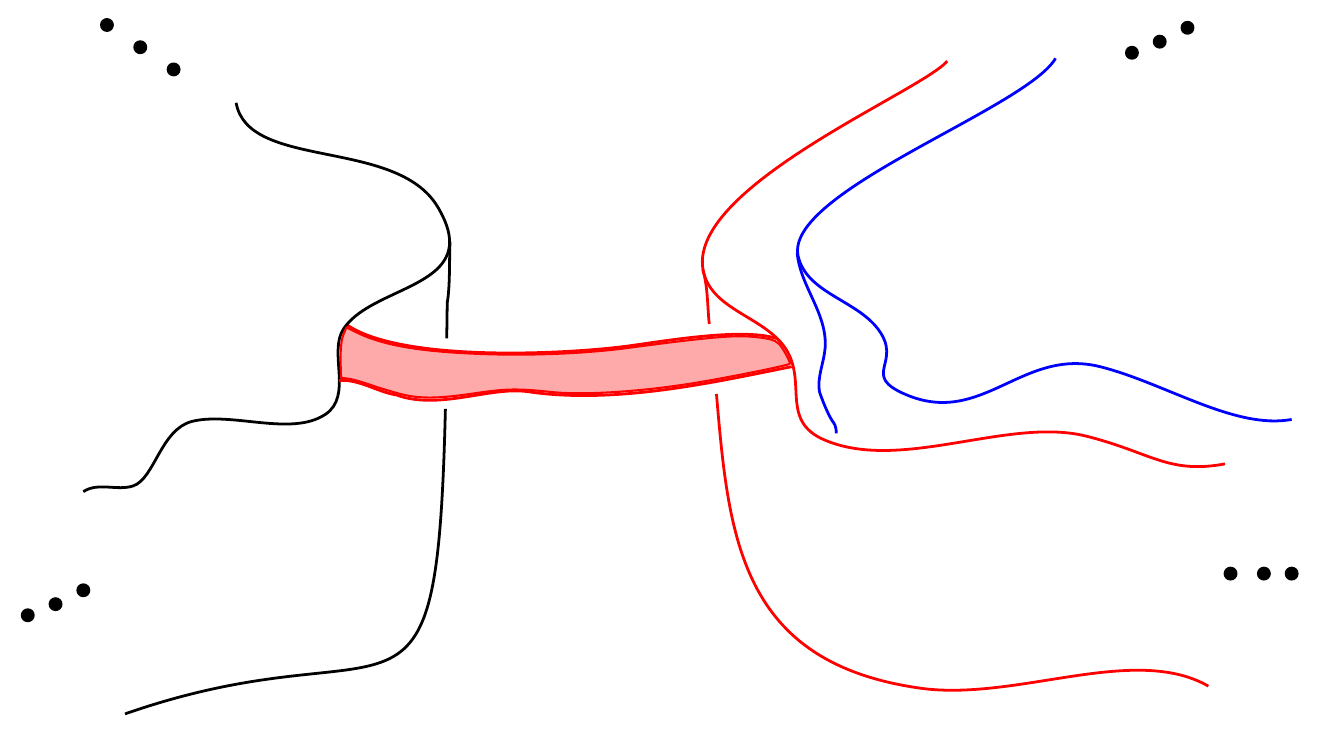}
\put(-2.3,0.3){$F_1$}
\put(-2.2,1.5){$\gamma_1$}
\put(-0.6,1){$\gamma_2$}
\put(-0.1,0.6){$F_2$}
\put(-1.75,1){$\gamma'_1$}
\put(-1.1,0.3){$F'_1$}

\caption{Taking a parallel copy of $\gamma_2$ and performing a band sum with $\gamma_1$.}
\end{figure}

However, it is not guaranteed that $F'_i \cap \gamma'_j = \phi$ for $i\neq j$, and also it is not guaranteed that $F'_1$ is a Seifert surface for $\gamma'_1$, since the band could have went through Seifert surfaces. We can fix this by altering Seifert surfaces. When the band goes through $F'_1$ we will take out two disks from $F'_1$ and attach a cylinder which connects the two circles as in Figure $11$. When the band goes through either $F'_2$ or $F'_3$, we will perform tubing as in Figure $12$. After all the alterations, it is guaranteed that $F'_i \cap \gamma'_j = \phi$ for $i\neq j$ and $F'_1$ is a Seifert surface for $\gamma'_1$. Since, we have not introduced any new triple intersection points, we have $\bar\mu_{\{\gamma_1, \gamma_2, \gamma_3 \}}(123) = \bar\mu_{\{\gamma'_1, \gamma'_2, \gamma'_3 \}}(123)$ as desired.

\begin{figure}[h]
\centering
\includegraphics[width=5.5in]{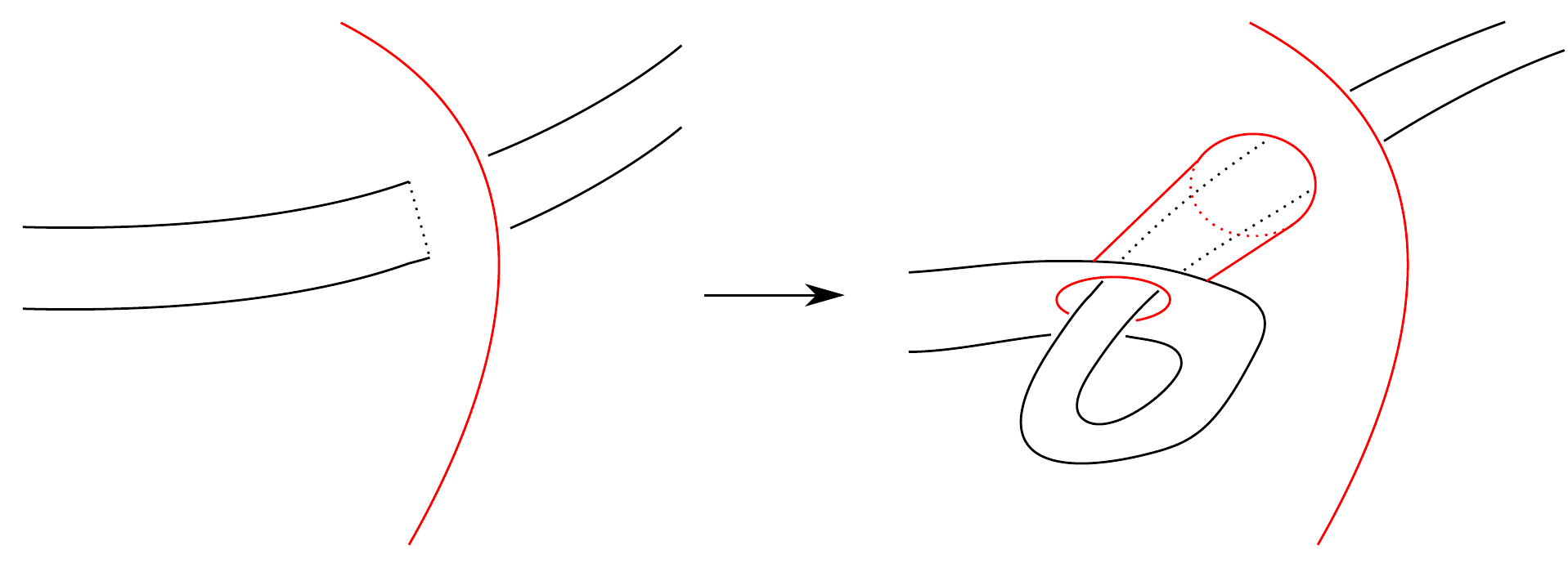}
\put(-5,1.03){band}
\put(-4.3,1.6){$F'_1$}
\put(-1.08,1.6){$F'_1$}
\put(-2.2,0.9){band}
\put(-1.9,1.45){cylinder}
\put(-3.1,0.7){alteration}

\caption{Alteration when $F'_1$ intersects the band}
\end{figure}

\begin{figure}[h]
\centering
\includegraphics[width=5.5in]{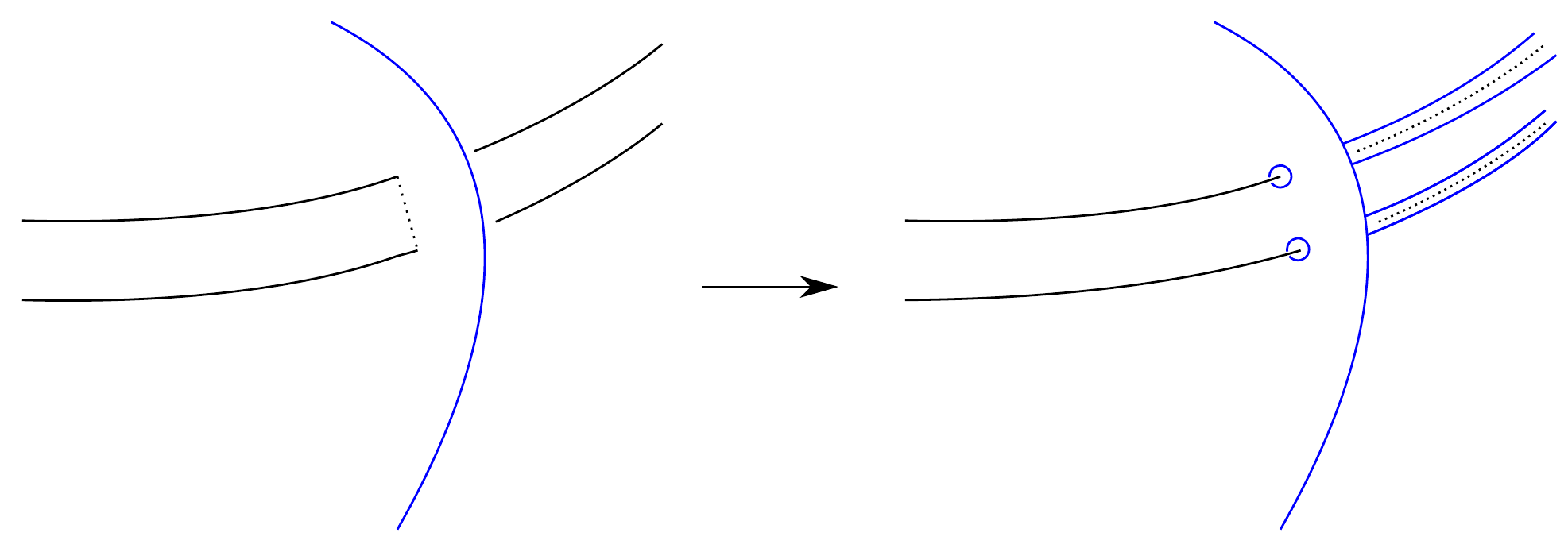}
\put(-5,1.03){band}
\put(-4.75,1.6){$F'_2$ or $F'_3$}
\put(-1.6,1.6){$F'_2$ or $F'_3$}
\put(-2,1){band}
\put(0,1.45){tube}
\put(-3.1,0.7){alteration}

\caption{Alteration when $F'_2$ or $F'_3$ intersects the band}
\end{figure}
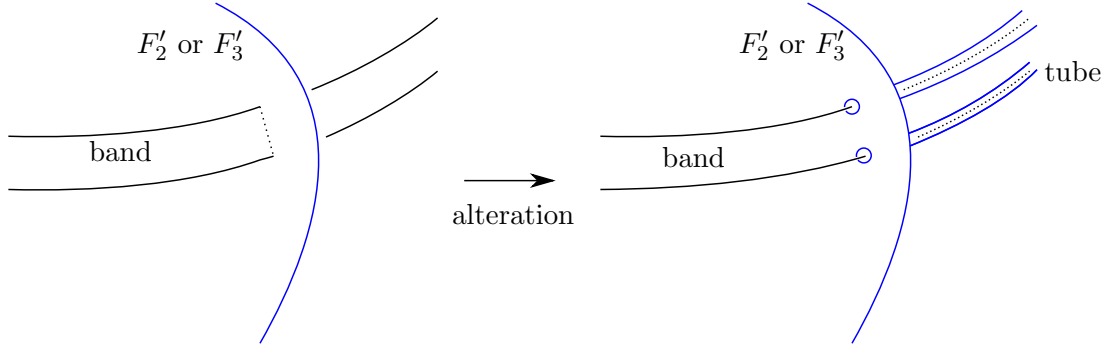

Notice that for any two derivative $\{ \gamma_1, \gamma_2, \gamma_3 \}$ and $\{ \gamma'_1, \gamma'_2, \gamma'_3 \}$ associated to $H$, we can perform several moves as above (i.e. taking a parallel copy of one component and performing a band sum with other component on the Seifert surface $F$) with possibly changing orientations to get from one to the other. While performing such moves we observed that up to change of sign, their Milnor's triple linking number does not change. Hence, we can conclude that $S_{K,H} \subseteq S_{K,H,\{ b_1, b_2, b_3 \} } \cup -S_{K,H,\{ b_1, b_2, b_3 \} }$ as needed.
\end{proof}

Now, we state the main theorem.

\begin{theorem} Let $K$ be an algebraically slice knot with genus three Seifert surface $F$. Suppose $H$ is a metabolizer of $K$ and $\{ a_1, a_2, a_3, b_1, b_2, b_3 \}$ is a symplectic basis for $H_1(F)$ where $\spn(b_1,b_2,b_3) = H$. Let $$M = (m_{ij})_{6\times 6} =
 \begin{pmatrix}
  * & a & * & x_1 & * & y_1 \\
  a-1 & 0 & x_2 & 0 & y_2 & 0 \\
  * & x_2 & * & b & * & z_1 \\
  x_1 & 0 & b-1 & 0 & z_2 & 0 \\
  * & y_2 & * & z_2 & * & c \\
  y_1 & 0 & z_1 & 0 & c-1 & 0
 \end{pmatrix}$$ be the Seifert matrix with respect to the basis $\{ a_1, b_1, a_2, b_2, a_3, b_3 \}$. Then $S_{K,H}  \supseteq S_{K,H,\{ b_1, b_2, b_3 \} } \supseteq \{n\cdot((a-1)(b-1)(c-1)-abc+x_1x_2+y_1y_2+z_1z_2)|n\in \mathbb{Z} \}$.
\end{theorem}
\begin{proof}
Let $K$ be a given algebraically slice knot. By Lemma $4.1$ we can assume $K$ is the knot described in Figure $21$ at the end of this paper, which has the simplest Seifert surface $\Sigma$ with Seifert matrix $M$. Notice that since $\{\gamma_1, \gamma_2, \gamma_3 \}$ is the unlink,  $\bar\mu_{\{\gamma_1, \gamma_2, \gamma_3 \}}(123) = 0$, where $\{\gamma_1, \gamma_2, \gamma_3 \}$ is the derivative of $K$ associated with $H$ described in Figure $21$.

Now, we will produce links $L_n$ which is a derivative associated with $H$ such that $\bar\mu_{L_n}(123) = n \cdot ((a-1)(b-1)(c-1)-abc+x_1x_2+y_1y_2+z_1z_2)$ for each positive integer $n$. We will start with the case when $n=1$. Let $\gamma_{1,1}, \gamma_{1,2}, \gamma_{1,3}$ be embedding of circles on the surface $\widetilde{\Sigma_1}$ as described in Figure $22,23,$ and $24$ at the end of this paper. We will denote $\gamma'_{1,i}$ as the intersection dual of $\gamma_{1,i}$ for $i=1,2,3$. Let $\phi : \widetilde{\Sigma_1} \rightarrow \Sigma$ be the obvious map which sends the core of $i$th band of $\widetilde{\Sigma_1}$ to the core of the $i$th band of $\Sigma$. By abusing notation we will denote image of $\gamma_{1,i}$ as $\gamma_{1,i}$ and image of $\gamma'_{1,i}$ as $\gamma'_{1,i}$, for $i=1,2,3$. Then notice that for $i=1,2,3$, $[\gamma'_{1,i}]=a_i$ and $[\gamma_{1,i}]=b_i$ in $H_1(\Sigma)$ and $\{\gamma_{1,1}, \gamma_{1,2}, \gamma_{1,3}\}$ is a derivative of $K$ associated with $H$. Let $L_1$ be the link $\{\gamma_{1,1}, \gamma_{1,2}, \gamma_{1,3}\} \subset S^3$ on the Seifert surface $\Sigma$.

Now, we will show that $\bar\mu_{L_1}(123) = ((a-1)(b-1)(c-1)-abc+x_1x_2+y_1y_2+z_1z_2)$. Notice that $\bar\mu_{L_1}(123)$ is equal to $\bar\mu_{\gamma_{1,2}, \gamma_{1,3}, \gamma_{1,1}}(123)$ by the definition so we will calculate $\bar\mu_{\gamma_{1,2}, \gamma_{1,3}, \gamma_{1,1}}(123)$ instead. Let $\pi$ be $\pi_1(S^3 - L_1)$, $\lambda_{1,1}, \lambda_{1,2}, \lambda_{1,3}$ be longitudes of $\gamma_{1,1}, \gamma_{1,2}, \gamma_{1,3}$ respectively, and $\mu_{1,1}, \mu_{1,2}, \mu_{1,3}$ be meridians of $\gamma_{1,1}, \gamma_{1,2}, \gamma_{1,3}$ respectively. Then $[\lambda_{1,1}] \in \pi_2$, since $\lk(\gamma_{1,i},\gamma_{1,j}) = 0$, for $i\neq j$. Hence by Proposition $2.2$, it is not necessary to specify the basing of $\lambda_{1,1}$, for the calculation of $\bar\mu_{\gamma_{1,2}, \gamma_{1,3}, \gamma_{1,1}}(123)$. Suppose $\lambda_{1,1}$ bounds a Seifert surface $\Sigma_1$ which does not intersect $\gamma_{1,2}$ and $\gamma_{1,3}$. For $i=1,2, \cdots, k_1$, where $k_1$ is the genus of $\Sigma_1$, let $\varphi_i$ and $\psi_i$ be core of the $2i-1$th band and $2i$th band respectively (see Figure $13$). For $i=1,2, \cdots, k_1$, let $c_i = [\varphi_i]$ and $d_i = [\psi_i]$ in $\pi_1(S^3 - L_1)$, then notice that $[\lambda_{1,1}] = \prod\limits_{i=1}^{k_1}{[c_i,d_i]} \in \pi_2 / \pi_3$. Then by Proposition $2.3$ we only need to calculate exponent sum of $[\mu_{1,2}]$ and $[\mu_{1,3}]$ occurring in $c_i$ and $d_i$, for all $i=1,2, \cdots, k_1$. Notice that basing of $\varphi_i$ and $\psi_i$ does not matter since we are only interested in their exponent sum of $[\mu_{1,2}]$ and $[\mu_{1,3}]$. Hence, we need to find a Seifert surface $\Sigma_1$ for $\lambda_{1,1}$ which does not intersect $\gamma_{1,2}$ and $\gamma_{1,3}$. We will use the obvious surface $\Sigma_1$ which bounds $\lambda_{1,1}$ and we will push it to the negative direction of $\Sigma$ (see Figure $14$). The problem with this surface is that it intersects $\gamma_{1,1}, \gamma_{1,2}$, and  $\gamma_{1,3}$. In order to fix this problem we need to perform several modifications, which are similar to the modifications performed in the proof of Lemma $4.2$, on the surface $\Sigma_1$.

\begin{figure}[h]
\centering
\includegraphics[width=5.5in]{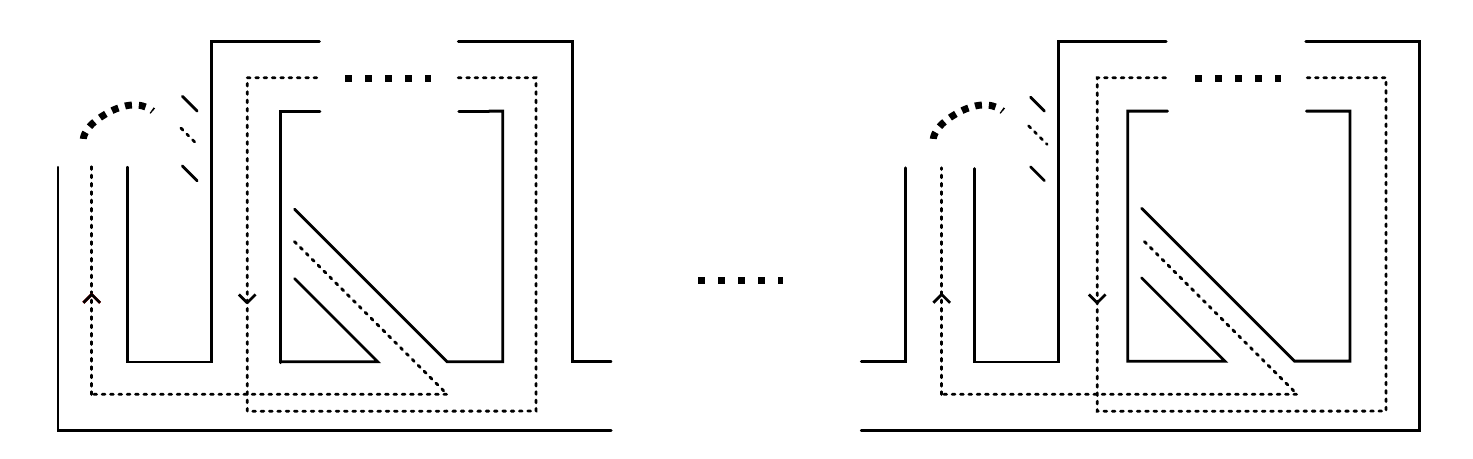}
\put(-5,-0.1){$\varphi_1$}
\put(-4,-0.1){$\psi_1$}
\put(-1.8,-0.1){$\varphi_{k_1}$}
\put(-0.8,-0.1){$\psi_{k_1}$}
\caption{Label of core of the band of Seifert surface $\Sigma_1$ for $\lambda_{1,1}$}
\end{figure}

\begin{figure}[h]
\centering
\includegraphics[width=5.5in]{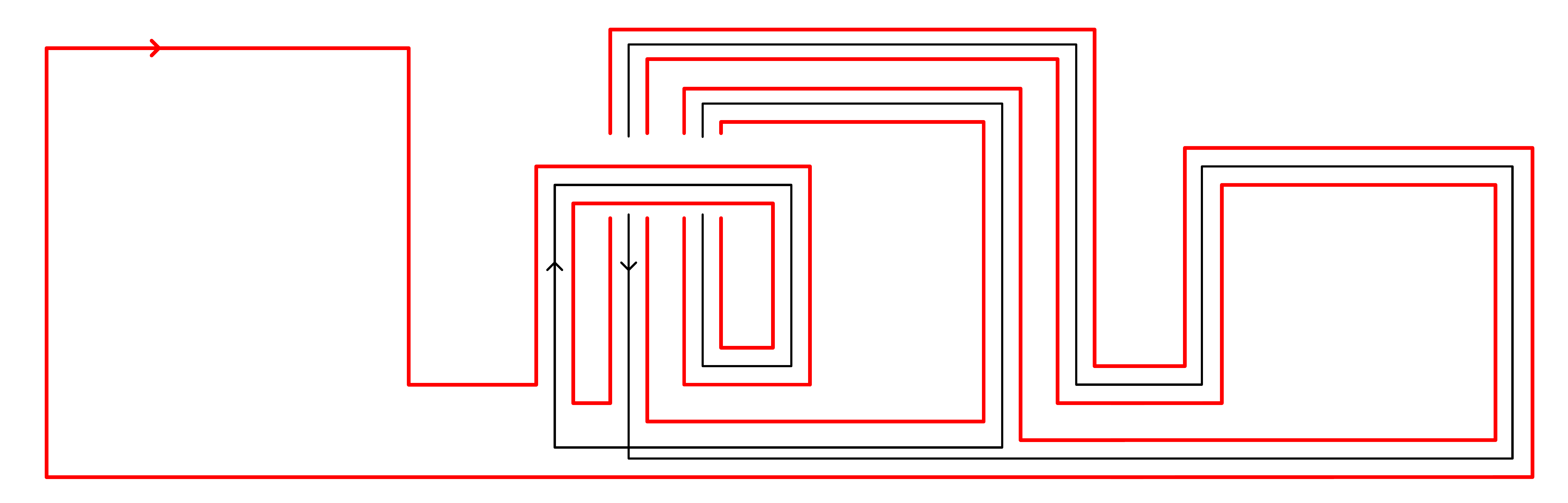}
\put(-3.55,1.5){$\psi_5$}
\put(-3.8,0.8){$\varphi_5$}
\caption{Seifert surface $\Sigma_1$ before modification and core of two bands $\varphi_5$ and $\psi_5$}
\end{figure}

First of all, $\Sigma_1$ intersects $\gamma_{1,1}$ only when the third band or the fifth band wraps around the second band. In this case we will drill out two disks from $\Sigma_1$ and connect them using a cylinder as in Figure $15$ to modify the surface $\Sigma_1$. Notice that whenever we do this we can use $\varphi$ in Figure $15$ as one of the core of the band. Since it has zero exponent sum of $[\mu_{1,2}]$ and $[\mu_{1,3}]$ occurring in $[\varphi]$, we do not have to worry about the case when $\Sigma_1$ intersects $\gamma_{1,1}$ for the calculation of $\bar\mu_{L_1}(123)$.

\begin{figure}[h]
\centering
\includegraphics[width=5.5in]{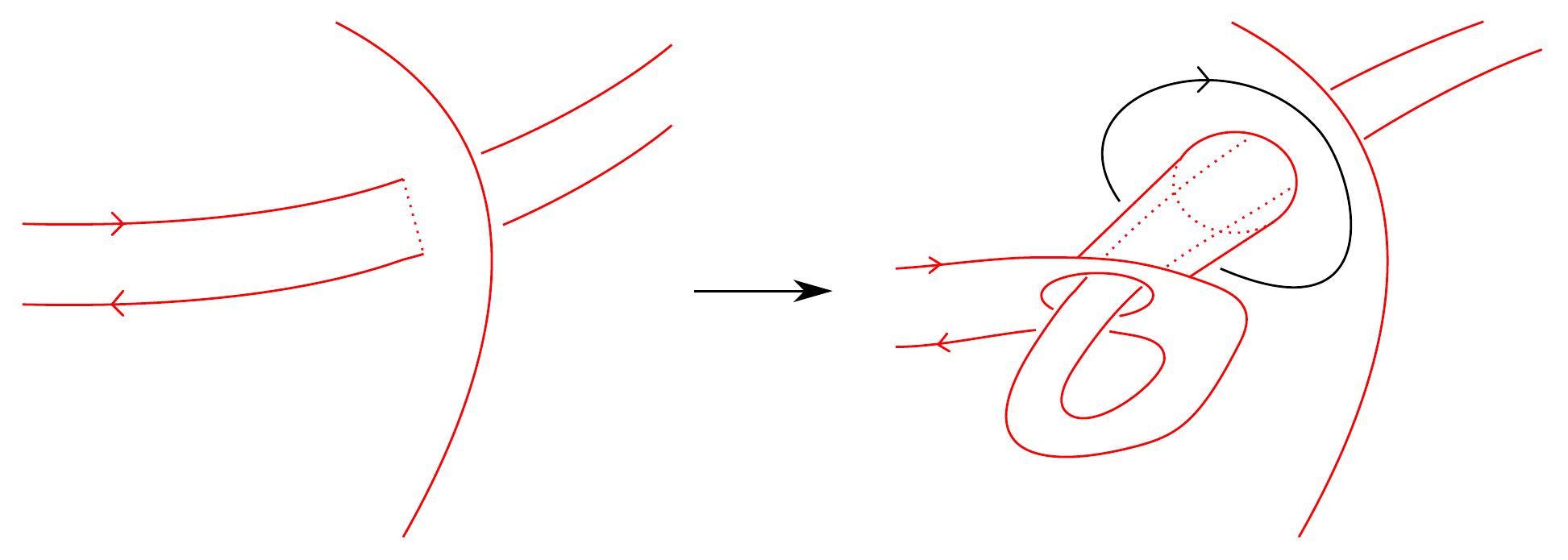}
\put(-5,1.3){$\gamma_{1,1}$}
\put(-4,1.8){$\Sigma_1$}
\put(-2.1,1.15){$\gamma_{1,1}$}
\put(-0.95,1.8){$\Sigma_1$}
\put(-1.5,1.7){$\varphi$}
\caption{Modification of $\Sigma_1$ when $\gamma_{1,1}$ intersects $\Sigma_1$}
\end{figure}
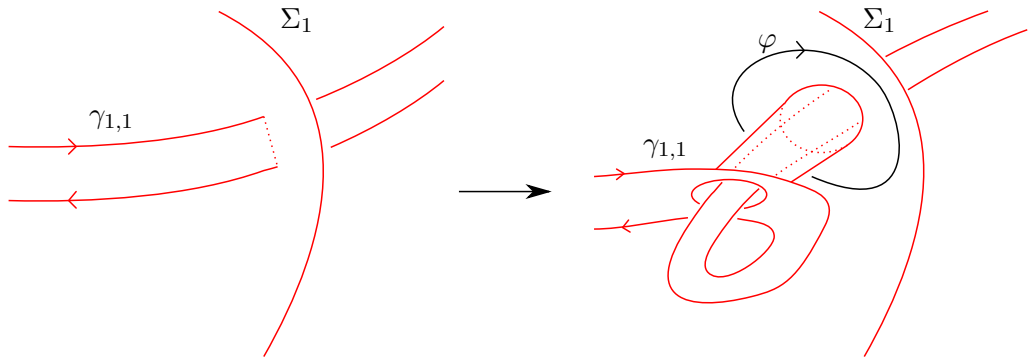

When the first band goes through the surface $\Sigma_1$, we will tube along the $\gamma_{1,2}$ and $\gamma_{1,3}$, where $\Sigma_1$ intersects $\gamma_{1,2}$ and $\gamma_{1,3}$, as in Figure $16$. We will let $\varphi_1, \varphi_2, \psi_1, \psi_2$ be the circles described in Figure $16$. Now, we need to calculate exponent sum of $[\mu_{1,2}]$ and $[\mu_{1,3}]$ occurring in $[\varphi_1], [\varphi_2], [\psi_1], [\psi_2]$. For $[\varphi_{1}]$, it has one exponent sum of $[\mu_{1,2}]$, and for $[\varphi_{2}]$, it has one exponent sum of $[\mu_{1,3}]$. Therefore, we only need to calculate exponent sum of $[\mu_{1,3}]$ for $[\psi_1]$, and $[\mu_{1,2}]$ for $[\psi_2]$. Notice that we can think of $\psi_1$ as positive push off of a circle on the Seifert surface $\Sigma$ which takes $-a_3$ value in $H_1(\Sigma)$ and we can think of $\psi_2$ as negative push off of a circle  on the Seifert surface $\Sigma$ which takes $b_1+a_2+b_2+b_3$ value in $H_1(\Sigma)$. Then $\psi_1$ has $(c-1)$ exponent sum of $[\mu_{1,3}]$ and $\psi_2$ has $-b$ exponent sum of $[\mu_{1,2}]$ by the following calculations :

\begin{equation*}
\begin{split}
   \lk(\psi_1,\gamma_{1,3})=& 
   \begin{pmatrix}
  0 & 0 & 0 & 0 & 0 & 1
 \end{pmatrix} \cdot 
 M \cdot
 {\begin{pmatrix}
 0 & 0 & 0 & 0 & -1 & 0 
 \end{pmatrix}}^T\\
 =& -(c-1)\\
 \lk(\psi_2,\gamma_{1,2})=& 
   \begin{pmatrix}
  0 & 1 & 1 & 1 & 0 & 1
 \end{pmatrix} \cdot 
 M \cdot
 {\begin{pmatrix}
 0 & 0 & 0 & 1 & 0 & 0 
 \end{pmatrix}}^T\\
 =& b.\\
\end{split}
\end{equation*}
Then in total we have $-(c-1)-b$ exponent sum of $[[\mu_{1,2}],[\mu_{1,3}]]$ in $[\varphi_1,\psi_1] \cdot [\varphi_2,\psi_2] \in \pi_2/\pi_3$, by Proposition $2.3$. Since the first band goes through the surface $(a-1)$ times in total we have $(a-1)(-(c-1)-b)$ effect on the Milnor's triple linking number. 

\begin{figure}[h]
\centering
\includegraphics[width=4.5in]{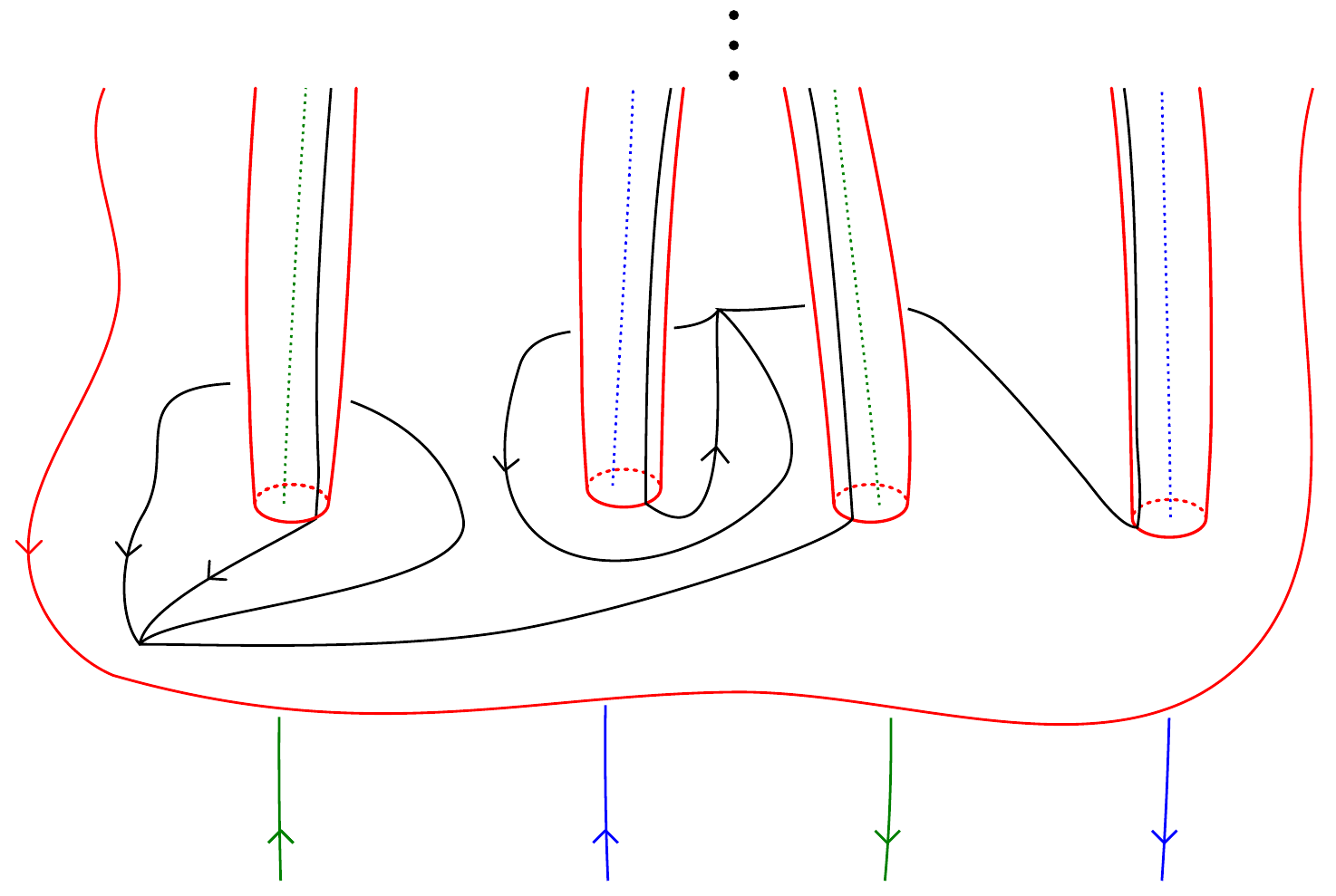}
\put(-4.4,2){$\Sigma_1$}
\put(-3,0.7){$\psi_1$}
\put(-4.2,1.6){$\varphi_1$}
\put(-3,1.7){$\varphi_2$}
\put(-1,1.7){$\psi_2$}
\put(-3.9,0.3){$\gamma_{1,2}$}
\put(-2.8,0.3){$\gamma_{1,3}$}
\caption{$\psi_1, \varphi_1, \psi_2,$ and $\varphi_2$}
\end{figure}
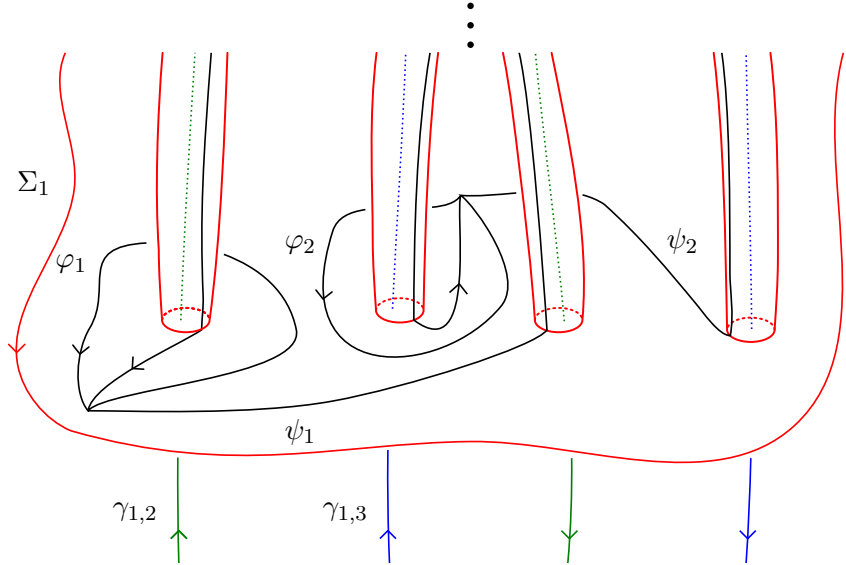

For the case where the third band goes through the surface $\Sigma_1$, we will use the same method. We will tube along the $\gamma_{1,3}$ as before (see Figure $17$). We will let $\varphi_3, \psi_3$ be the circles described in Figure $17$. Same as before we calculate exponent sum of $[\mu_{1,2}]$ and $[\mu_{1,3}]$ occurring in $[\varphi_3]$ and $[\psi_3]$, respectively. $[\varphi_{3}]$ has only one exponent sum of $[\mu_{1,3}]$, so for $[\psi_{3}]$ we only need to calculate exponent sum of $[\mu_{1,2}]$. We can think of $\psi_3$ as positive push off of a circle on the Seifert surface $\Sigma$ which takes $-a_1+b_1-b_3$ value in $H_1(\Sigma)$, hence $\psi_3$ has $-x_1$ exponent sum of $[\mu_{1,2}]$ by the following calculations :

\begin{equation*}
\begin{split}
   \lk(\psi_3,\gamma_{1,2})=& 
   \begin{pmatrix}
  0 & 0 & 0 & 1 & 0 & 0
 \end{pmatrix} \cdot 
 M \cdot
 {\begin{pmatrix}
 -1 & 1 & 0 & 0 & 0 & -1 
 \end{pmatrix}}^T\\
 =& -x_1\\
\end{split}
\end{equation*}
Then again by Proposition $2.3$, in total we have $x_1$ exponent sum of $[[\mu_{1,2}],[\mu_{1,3}]]$ in $[\varphi_3,\psi_3] \in \pi_2/\pi_3$. Since the third band goes through the surface $x_2$ times in total we have $x_1x_2$ effect on the Milnor's triple linking number. 

We use the same method for the case when the fifth band goes through the surface. We will tube along the $\gamma_{1,2}$ as before (see Figure $17$). We will let $\varphi_4, \psi_4$ be the circles described in Figure $17$. $[\varphi_{4}]$ has only one exponent sum of $[\mu_{1,2}]$, so for $[\psi_{4}]$ we only need to calculate exponent sum of $[\mu_{1,3}]$. We can think of $\psi_4$ as positive push off of a circle on the Seifert surface $\Sigma$ which takes $a_1-b_1+b_3$ value in $H_1(\Sigma)$ hence $\psi_4$ has $y_1$ exponent sum of $[\mu_{1,3}]$ by the following calculations :

\begin{equation*}
\begin{split}
   \lk(\psi_3,\gamma_{1,2})=& 
   \begin{pmatrix}
  0 & 0 & 0 & 0 & 0 & 1
 \end{pmatrix} \cdot 
 M \cdot
 {\begin{pmatrix}
 1 & -1 & 0 & 0 & 0 & 1 
 \end{pmatrix}}^T\\
 =& y_1\\
\end{split}
\end{equation*}
Then again by Proposition $2.3$, in total we have $y_1$ exponent sum of $[[\mu_{1,2}],[\mu_{1,3}]]$ in $[\varphi_4,\psi_4] \in \pi_2/\pi_3$. Since the third band goes through the surface $y_2$ times in total we have $y_1y_2$ effect on the Milnor's triple linking number. 

\begin{figure}[h]
\centering
\includegraphics[width=3in]{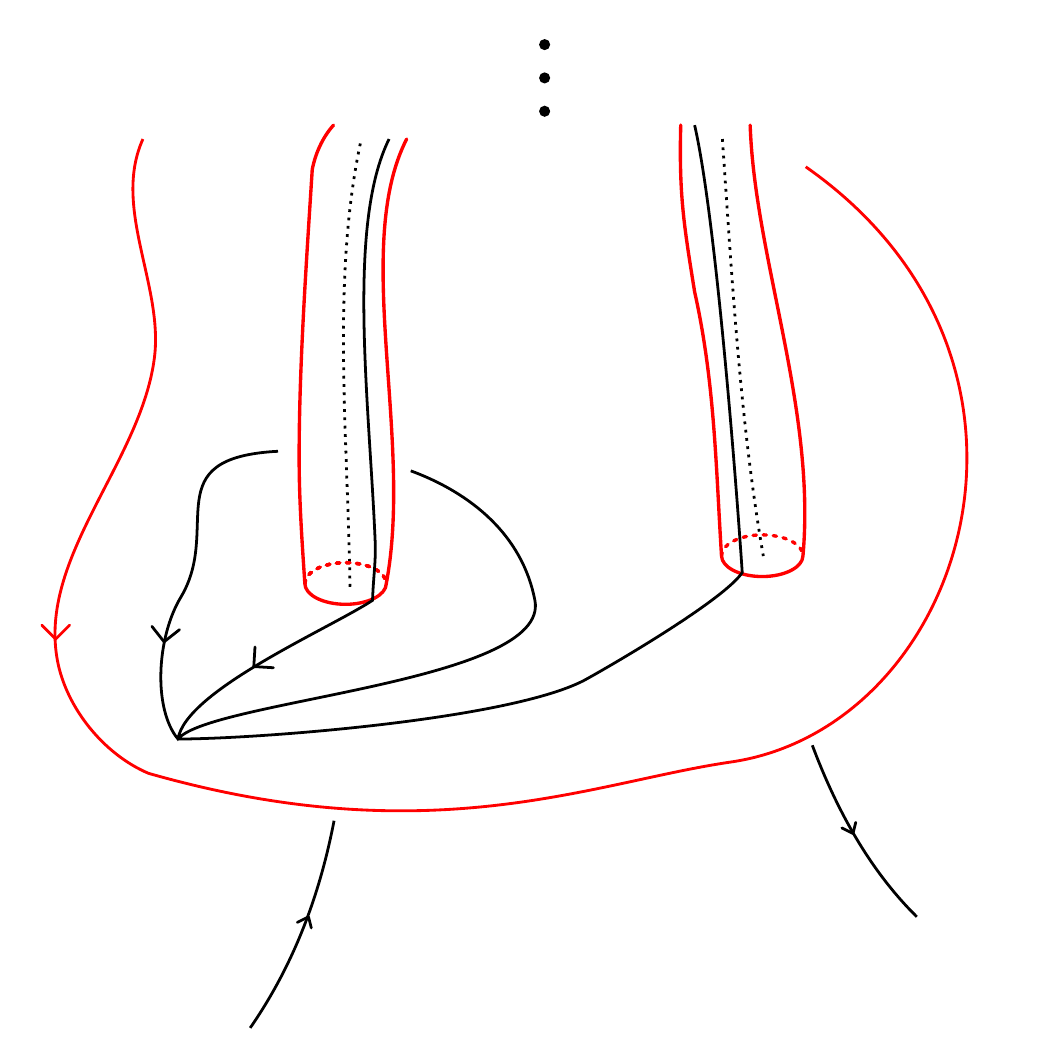}
\put(-2.8,2){$\Sigma_1$}
\put(-1.3,1){$\psi_i$}
\put(-1.6,1.6){$\varphi_i$}
\put(-2,0.3){$\gamma_{1,2}$ or $\gamma_{1,3}$}
\caption{$\psi_i$ and $\varphi_i$ for $i=3,4$}
\end{figure}
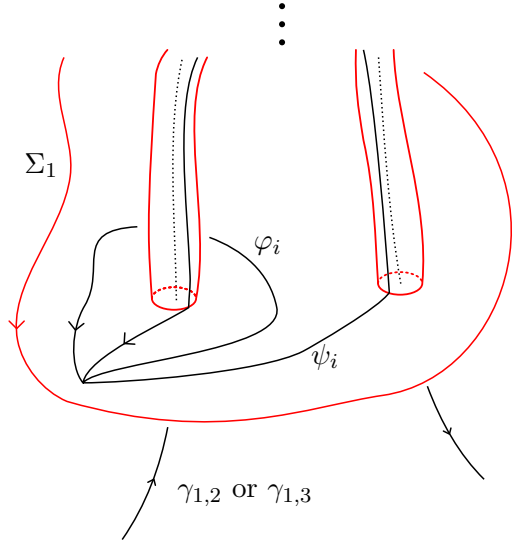

In addition, we have $\varphi_5,\psi_5$, which is the core of the two bands of $\Sigma_1$ before the modifications (see Figure $14$). We have pushed the surface $\Sigma_1$ towards the negative directions hence we can think of $\varphi_5$ as negative push off of a circle on the Seifert surface $\Sigma$ which takes $a_2+b_2$ value in $H_1(\Sigma)$ and we can think of $\psi_5$ as negative push off of a circle  on the Seifert surface $\Sigma$ which takes $-b_2-a_3$ value in $H_1(\Sigma)$. Then $\varphi_5$ has $b$ exponent sum of $[\mu_{1,2}]$ and $z_1$ exponent sum of $[\mu_{1,3}]$, and $\psi_5$ has $-z_2$ exponent sum of $[\mu_{1,2}]$ and $-c$ exponent sum of $[\mu_{1,3}]$ by the following calculations : 
\begin{equation*}
\begin{split}
   \lk(\varphi_5,\gamma_{1,2})=& 
   \begin{pmatrix}
  0 & 0 & 1 & 1 & 0 & 0
 \end{pmatrix} \cdot 
 M \cdot
 {\begin{pmatrix}
 0 & 0 & 0 & 1 & 0 & 0 
 \end{pmatrix}}^T\\
 =& b\\
 \lk(\varphi_5,\gamma_{1,3})=& 
   \begin{pmatrix}
  0 & 0 & 1 & 1 & 0 & 0
 \end{pmatrix} \cdot 
 M \cdot
 {\begin{pmatrix}
 0 & 0 & 0 & 0 & 0 & 1 
 \end{pmatrix}}^T\\
 =& z_1\\
 \lk(\psi_5,\gamma_{1,2})=& 
   \begin{pmatrix}
  0 & 0 & 0 & -1 & -1 & 0
 \end{pmatrix} \cdot 
 M \cdot
 {\begin{pmatrix}
 0 & 0 & 0 & 1 & 0 & 0 
 \end{pmatrix}}^T\\
 =& -z_2\\
 \lk(\psi_5,\gamma_{1,3})=& 
   \begin{pmatrix}
  0 & 0 & 0 & -1 & -1 & 0
 \end{pmatrix} \cdot 
 M \cdot
 {\begin{pmatrix}
 0 & 0 & 0 & 0 & 0 & 1 
 \end{pmatrix}}^T\\
 =& -c.\\
\end{split}
\end{equation*}
Using Proposition $2.3$, in total we have $-bc+z_1z_2$ exponent sum of $[[\mu_{1,2}],[\mu_{1,3}]]$ in $[\varphi_5,\psi_5] \in \pi_2/\pi_3$. 
If we total all the effects we get $(a-1)(-(c-1)-b)+x_1x_2+y_1y_2-bc+z_1z_2 = (a-1)(b-1)(c-1)-abc+x_1x_2+y_1y_2+z_1z_2$ as desired.

Now for integer $n$ greater than $1$, we will produce a link $L_n$ which is a derivative associated with $H$ such that $\bar\mu_{L_n}(123) = n \cdot ((a-1)(b-1)(c-1)-abc+x_1x_2+y_1y_2+z_1z_2)$. First we will describe $\widetilde{\Sigma_n}$ and $\gamma_{n,1}, \gamma_{n,2}, \gamma_{n,3}$ embedding of circles on $\widetilde{\Sigma_n}$. In order to do so we will start with $\gamma_{1,1}, \gamma_{1,2}, \gamma_{1,3}$ embedded in $\widetilde{\Sigma_1}$ and we will make some modifications to it. We will use Figure $22, 23$, and $24$ to describe the modifications. Let $\widetilde{\Sigma_n}$ be the same surface as $\widetilde{\Sigma_1}$. For Figure $22$, and Figure $24$ we will take $n$ parallel copies of green curves without changing anything else. For Figure $23$ we will alter the core of the bands which bound $\gamma_{1,1}$. The cores will wrap around the third and the fourth band of $\widetilde{\Sigma_n}$ as described in Figure $18$. Further inside the red band we will make some alterations as in Figure $19$. We will denote red circle, green circle, and blue circle on $\widetilde{\Sigma_n}$  as $\gamma_{n,1}, \gamma_{n,2},$ and $\gamma_{n,3}$ respectively. Also, we will denote $\gamma'_{n,i}$ as the intersection dual of $\gamma_{n,i}$, for $i=1,2,3$. For the convenience of readers we have presented $\gamma_{2,1}, \gamma_{2,2}, \gamma_{2,3}$ in Figure $25, 26,$ and $27$ at the end of this paper. As before let $\phi_n : \widetilde{\Sigma_n} \rightarrow \Sigma$ be the obvious map which sends the core of $i$th band of $\widetilde{\Sigma_n}$ to the core of the $i$th band of $\Sigma$. By abusing notation we will denote image of $\gamma_{n,i}$ under $\phi_n$ as $\gamma_{n,i}$ and image of $\gamma'_{n,i}$ under $\phi_n$ as $\gamma'_{n,i}$ for $i=1,2,3$. Then it is easy to check that for $i=1,2,3$, $[\gamma'_{n,i}]=a_i$ and $[\gamma_{n,i}]=b_i$ in $H_1(\Sigma)$ and $\{\gamma_{n,1}, \gamma_{n,2}, \gamma_{n,3}\}$ is a derivative of $K$ associated with $H$. As before we will let $L_n$ be the link $\{\gamma_{n,1}, \gamma_{n,2}, \gamma_{n,3}\} \subset S^3$ on the Seifert surface $\Sigma$. We will denote the obvious surface that $\gamma_{n,1}$ bounds by $\Sigma_n$. 

\begin{figure}[h]
\centering
\includegraphics[width=2.5in]{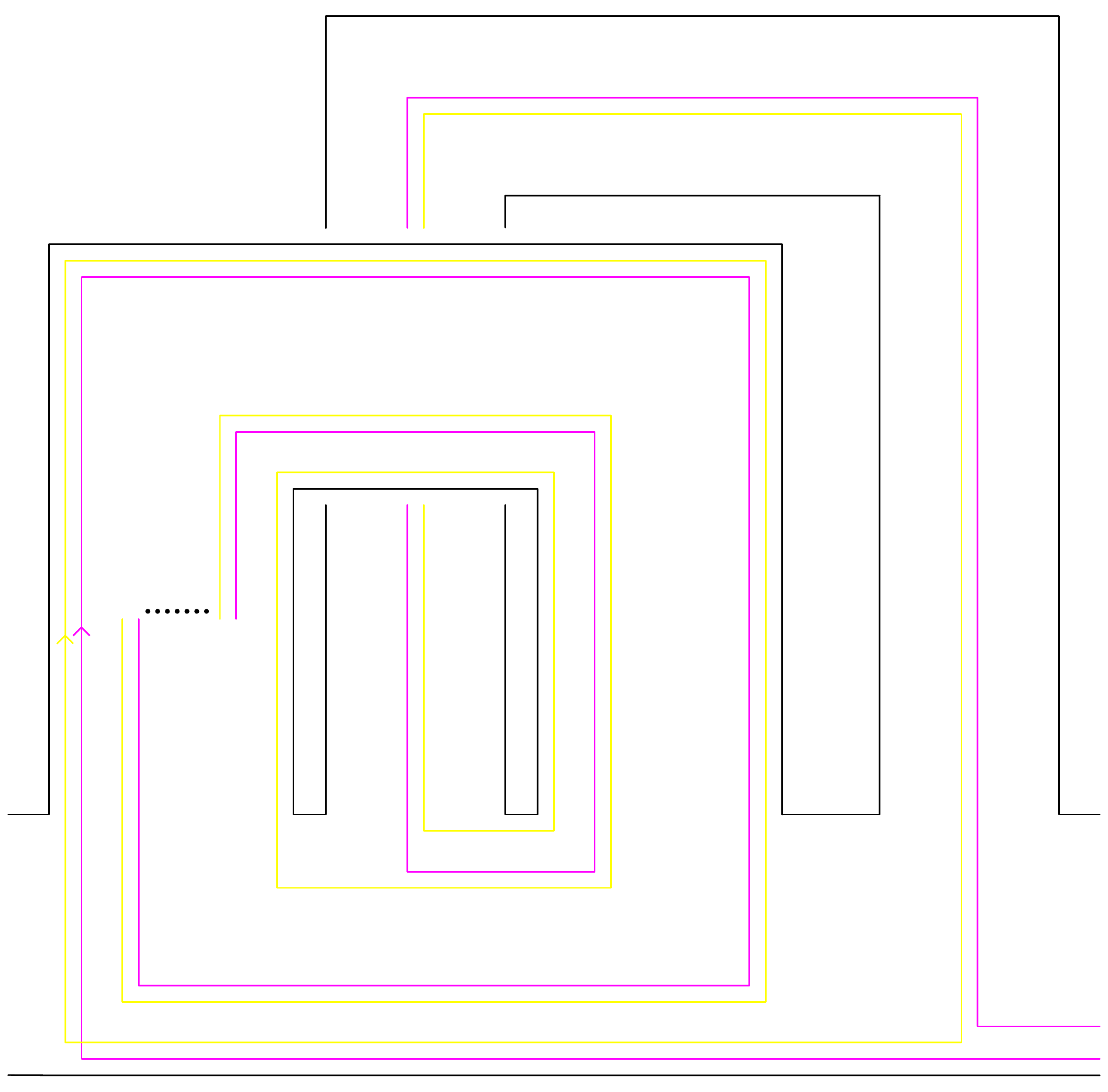}
\put(-2.3,1.17){\scriptsize{$n-2$}}
\caption{The yellow curve represents core of the first band $\varphi_{n,k_n}$ and the purple curve represents core of the second band $\psi_{n,k_n}$.}
\end{figure}

\begin{figure}[h]
\centering
\includegraphics[width=2.5in]{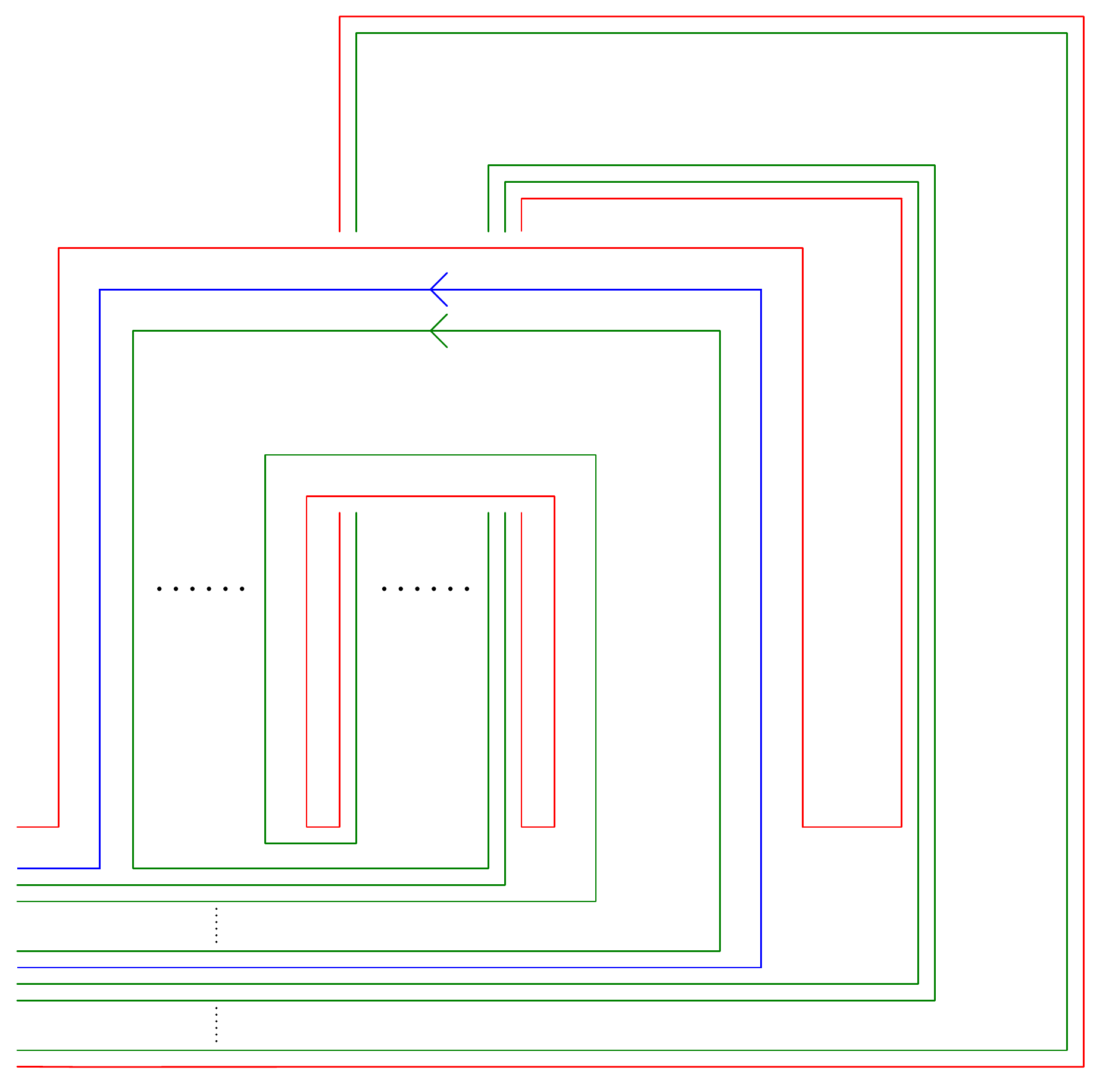}
\put(-2.18,1.17){\scriptsize{$n-1$}}
\caption{The green curves represents $\gamma_{n,2}$ and the blue curve represents $\gamma_{n,3}$ inside two bands which bounds $\gamma_{n,1}$.}
\end{figure}

For the calculation of $\bar\mu_{L_n}(123)$ we will omit some of the details since they are very similar to the case when $n=1$. We will modify $\Sigma_n$ so that it does not intersect $\gamma_{n,1}, \gamma_{n,2},$ and $\gamma_{n,3}$. As before we will not worry about the case when $\Sigma_n$ intersects $\gamma_{n,1}$. When the first band goes through $\Sigma_n$, since we have taken $n$ parallel copies of green curves the total effect by green curves on Milnor's triple linking number should be multiplied by $n$. Also since the blue curve, which we denote it by $\psi_{n,2}$ (which was $\psi_2$ for the case when $n=1$), represents $b_1+n\cdot a_2+b_2+b_3$ in $H_1(\Sigma)$, we have the following calculation:
\begin{equation*}
\begin{split}
 \lk(\psi_{n,2},\gamma_{1,2})=& 
   \begin{pmatrix}
  0 & 1 & n & 1 & 0 & 1
 \end{pmatrix} \cdot 
 M \cdot
 {\begin{pmatrix}
 0 & 0 & 0 & 1 & 0 & 0 
 \end{pmatrix}}^T\\
 =& n\cdot b.\\
\end{split}
\end{equation*}
This tells us that the total effect by the blue curve on Milnor's triple linking number should be multiplied by $n$ also. Hence in total we have $n\cdot (a-1)(-(c-1)-b)$ effect on the Milnor's triple linking number. It is easy to see that the effect of the third band and the fifth band going through the surface $\Sigma_n$ also needs to be multiplied by $n$, hence we have $n\cdot x_1x_2 + n\cdot y_1y_2$.

Lastly we have two more bands of $\Sigma_n$ before any modifications, $\varphi_{n,k_n}$ and $\psi_{n,k_n}$ where $k_n$ is genus of $\Sigma_n$. Again, we can think of $\varphi_{n,k_n}$ as negative push off of a circle on the Seifert surface $\Sigma$ which takes $n\cdot a_2+b_2$ value in $H_1(\Sigma)$ and we can think of $\psi_{n,k_n}$ as negative push off of a circle  on the Seifert surface $\Sigma$ which takes $-(n-1)\cdot a_2-b_2-a_3$ value in $H_1(\Sigma)$. Hence we have, $\varphi_{n,k_n}$ has $n\cdot b$ exponent sum of $[\mu_{1,2}]$ and $n \cdot z_1$ exponent sum of $[\mu_{1,3}]$, and $\psi_{n,k_n}$ has $-(n-1)\cdot b-z_2$ exponent sum of $[\mu_{1,2}]$ and $-(n-1)\cdot z_1-c$ exponent sum of $[\mu_{1,3}]$ by the following calculations : 
\begin{equation*}
\begin{split}
   \lk(\varphi_{n,k_n},\gamma_{1,2})=& 
   \begin{pmatrix}
  0 & 0 & n & 1 & 0 & 0
 \end{pmatrix} \cdot 
 M \cdot
 {\begin{pmatrix}
 0 & 0 & 0 & 1 & 0 & 0 
 \end{pmatrix}}^T\\
 =& n\cdot b\\
 \lk(\varphi_{n,k_n},\gamma_{1,3})=& 
   \begin{pmatrix}
  0 & 0 & n & 1 & 0 & 0
 \end{pmatrix} \cdot 
 M \cdot
 {\begin{pmatrix}
 0 & 0 & 0 & 0 & 0 & 1 
 \end{pmatrix}}^T\\
 =& n\cdot z_1\\
 \lk(\psi_{n,k_n},\gamma_{1,2})=& 
   \begin{pmatrix}
  0 & 0 & -(n-1) & -1 & -1 & 0
 \end{pmatrix} \cdot 
 M \cdot
 {\begin{pmatrix}
 0 & 0 & 0 & 1 & 0 & 0 
 \end{pmatrix}}^T\\
 =& -(n-1)\cdot b-z_2\\
 \lk(\psi_{n,k_n},\gamma_{1,3})=& 
   \begin{pmatrix}
  0 & 0 & -(n-1) & -1 & -1 & 0
 \end{pmatrix} \cdot 
 M \cdot
 {\begin{pmatrix}
 0 & 0 & 0 & 0 & 0 & 1 
 \end{pmatrix}}^T\\
 =& -(n-1)\cdot z_1 -c.\\
\end{split}
\end{equation*}
In total, we have $n\cdot(-bc+z_1z_2)$ exponent sum of $[[\mu_{1,2}],[\mu_{1,3}]]$ in $[\varphi_{n,k_n},\psi_{n,k_n}] \in \pi_2/\pi_3$, by Proposition $2.3$. 
If we total all the effects, we get $n\cdot ((a-1)(b-1)(c-1)-abc+x_1x_2+y_1y_2+z_1z_2)$ as desired.

Even though we have not covered the case when $n$ is a negative integer, this can be easily handled. On $\Sigma$ imagine we are sliding the third and the fourth band to the right so they pass the fifth and the sixth band. Then imagine mapping $\Sigma_n$ to it so that now the green circle represents $b_3$ in $H_1(\Sigma)$ and the blue circle represents $b_2$ in $H_1(\Sigma)$. Then we can follow the same calculations but since the roles of the green circle and the blue circle have been switched, we get the desired equation for negative integers.
\end{proof}

\begin{remark} Notice that once we fix a metabolizer $H$ of a knot, the set $\{n\cdot((a-1)(b-1)(c-1)-abc+x_1x_2+y_1y_2+z_1z_2)|n\in \mathbb{Z} \}$ does not depend on the choice of a basis of $H$ and the choice of a symplectic basis for $H_1(F)$ which extends that basis of $H$. This can be easily seen by considering a Seifert matrix $M' =
 \begin{pmatrix}
  A & B \\
  B^\top - Id & 0 
 \end{pmatrix}$ with respect to the basis $\{ a_1, a_2, a_3, b_1, b_2, b_3 \}$. Then notice that $(a-1)(b-1)(c-1)-abc+x_1x_2+y_1y_2+z_1z_2 = \det(B^\top - Id) - \det(B)$. If we pick a different basis, $\{\widetilde{b_1}, \widetilde{b_2}, \widetilde{b_3} \}$, for $H$ and extend it to a symplectic basis, $\{ \widetilde{a_1}, \widetilde{a_2}, \widetilde{a_3},\widetilde{b_1}, \widetilde{b_2}, \widetilde{b_3} \}$, for $H_1(F)$, then we get the Seifert matrix $\widetilde{M'} =
 \begin{pmatrix}
  \widetilde{A} & P^\top_2BP_1 \\
  P^\top_1(B^\top - Id)P_2 & 0 
 \end{pmatrix}$ with respect to the basis $\{ \widetilde{a_1}, \widetilde{a_2}, \widetilde{a_3},\widetilde{b_1}, \widetilde{b_2}, \widetilde{b_3} \}$, where $P_1$ and $P_2$ are change of basis matrices. Hence, we have $\det(P^\top_1(B^\top - Id)P_2) - \det(P^\top_2BP_1) = \pm \det(B^\top - Id) - \det(B)$, so we can conclude that the choice of a basis does not have an effect on the set $\{n\cdot((a-1)(b-1)(c-1)-abc+x_1x_2+y_1y_2+z_1z_2)|n\in \mathbb{Z} \}$.
\end{remark}

From this main theorem we have the first corollary which immediately follows from Theorem $4.3$. This corollary tells us that even the derivatives of the unknot have complicated Milnor's triple linking number.

\begin{corollary} Let $U$ be the unknot with a Seifert surface $F$ given in Figure $20$. For $i=1,2,3$, let $\alpha_i$ be the core of the $(2i-1)$th band and $\beta_i$ be the core of the $(2i)$th band. Let $H := \spn([\beta_1], [\beta_2], [\beta_3])$, then $S_{U,H} = \mathbb{Z}$.
\end{corollary}

\begin{proof}
Let $a_i=[\alpha_i] \in H_1(F)$ and $b_i=[\beta_i] \in H_1(F)$ for $i=1,2,3$. Then we have $$M =
 \begin{pmatrix}
  0 & 1 & 0 & 0 & 0 & 0 \\
  0 & 0 & 0 & 0 & 0 & 0 \\
  0 & 0 & 0 & 1 & 0 & 0 \\
  0 & 0 & 0 & 0 & 0 & 0 \\
  0 & 0 & 0 & 0 & 0 & 1 \\
  0 & 0 & 0 & 0 & 0 & 0
 \end{pmatrix}$$ a Seifert matrix with respect to the basis $\{a_1, b_1, a_2, b_2, a_3, b_3 \}$ and $H=\spn(b_1,b_2,b_3)$ is a metabolizer. By Theorem $4.3$ we have $S_{U,H} \supseteq \{n\cdot(0\cdot0\cdot0-1\cdot1\cdot1)|n\in \mathbb{Z} = \mathbb{Z}\} = \mathbb{Z}$ as desired.
\end{proof}

\begin{figure}[h]
\centering
\includegraphics[width=6in]{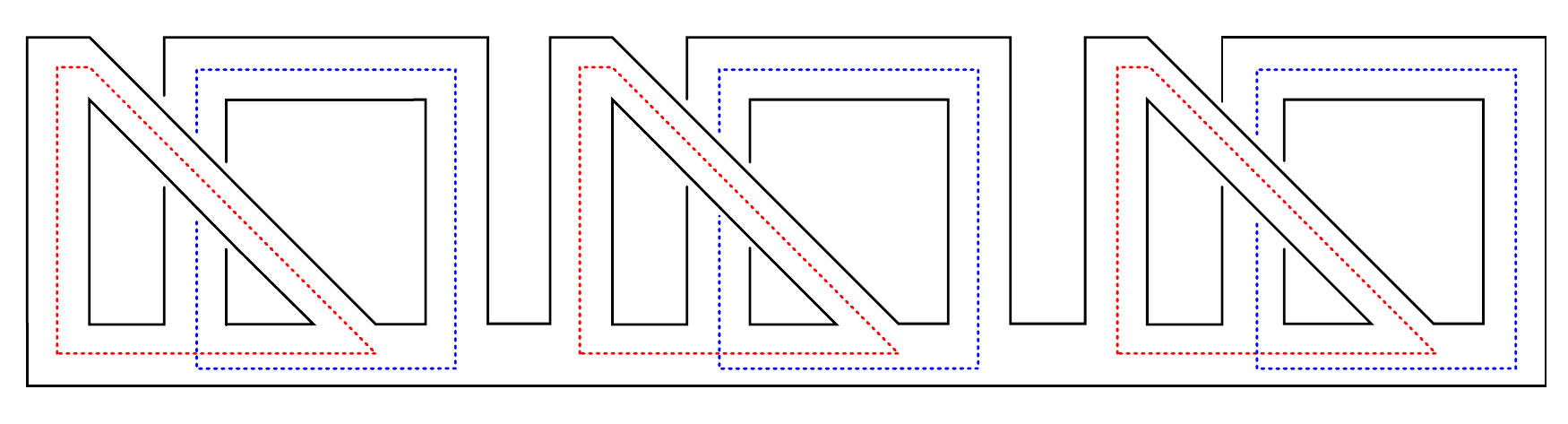}
\put(-5.6,0.5){$\alpha_1$}
\put(-4.6,0.5){$\beta_1$}
\put(-3.6,0.5){$\alpha_2$}
\put(-2.6,0.5){$\beta_2$}
\put(-1.55,0.5){$\alpha_3$}
\put(-0.55,0.5){$\beta_3$}
\put(-5.5,-0){$a_1$}
\put(-4.8,-0){$b_1$}
\put(-3.5,-0){$a_2$}
\put(-2.8,-0){$b_2$}
\put(-1.5,-0){$a_3$}
\put(-0.8,-0){$b_3$}
\caption{Disk-band form of Seifert surface $F$ for the unknot $U$}
\end{figure}
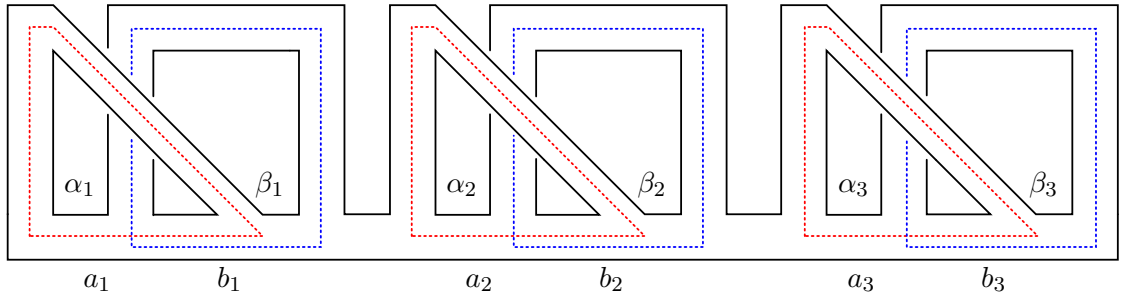
We have one more immediate corollary from Theorem $4.3$ which tells us that when there is a knot which is the connected sum of three genus one algebraically slice knots, it has at least one derivative which has non-zero Milnor's triple linking number.

\begin{corollary} Let $K_1, K_2, K_3$ be algebraically slice knots with genus one Seifert surfaces. Then $K=K_1 \# K_2 \# K_3$ has a derivative $\{ \gamma_1, \gamma_2, \gamma_3 \}$ where $\bar\mu_{\{ \gamma_1, \gamma_2, \gamma_3 \}}(123) \neq 0$.
\end{corollary}
\begin{proof}
It is enough to show that there exist a metabolizer $H$ such that $S_{K,H} \supseteq m\cdot\mathbb{Z}$ where $m\neq 0$.

Let $F_1,F_2,$ and $F_3$ be genus one Seifert surfaces for $K_1, K_2,$ and $K_3$ respectively. For $i=1,2,3$, let $M_i=\begin{pmatrix}
  d_i & e_i \\ 
  e_i-1 & 0 
 \end{pmatrix}$ be a Seifert matrix with respect to a symplectic basis $\{a_i, b_i\}$ of $H_1(F_i)$, so that $H_i : = \spn(b_i)$ is a metabolizer for $K_i$. For $i=1,2,3$, let $n_i = gcd(2e_i -1,-d_i), x_i = \frac{2e_i-1}{n_i},$ and $y_i = \frac{-d_i}{n_i}$, then since $gcd(x_i,y_i)=1$, there exist a pair of integers $(z_i,w_i)$ such that $-x_i w_i + z_i y_i = 1$ (i.e. $-(2e_i-1) w_i - d_iz_i = n_i$).
 
 Then for $i=1,2,3$, we have the following calculations : 
 \begin{equation*}
\begin{split}
   \begin{pmatrix}
  z_i & w_i 
 \end{pmatrix} \cdot 
 \begin{pmatrix}
 d_i & e_i \\
 e_i-1 & 0
 \end{pmatrix} \cdot
 \begin{pmatrix}
 x_i \\
 y_i
 \end{pmatrix}
 =& \frac{1}{n_i}\cdot \begin{pmatrix}
  z_i & w_i 
 \end{pmatrix} \cdot 
 \begin{pmatrix}
 d_i(2e_i-1)-e_id_i\\
 (e_i-1)(2e_i-1)
 \end{pmatrix}\\
= & \frac{1}{n_i} \cdot (d_i(2e_i-1)z_i-e_id_iz_i+(e_i-1)(2e_i-1)w_i)\\
= & \frac{1}{n_i} \cdot (-d_iz_i-(2e_i-1) w_i + e_i(d_iz_i+(2e_i-1) w_i))\\
= & \frac{1}{n_i} \cdot (n_i - n_ie_i)\\
= & 1-e_i\\
\begin{pmatrix}
  x_i & y_i 
 \end{pmatrix} \cdot 
 \begin{pmatrix}
 d_i & e_i \\
 e_i-1 & 0
 \end{pmatrix} \cdot
 \begin{pmatrix}
 z_i \\
 w_i
 \end{pmatrix}
 =& \frac{1}{n_i}\cdot \begin{pmatrix}
  2e_i-1 & -d_i 
 \end{pmatrix} \cdot 
 \begin{pmatrix}
 d_iz_i+e_iw_i\\
 (e_i-1)z_i
 \end{pmatrix}\\
= & \frac{1}{n_i} \cdot (d_iz_i (2e_i -1) + e_iw_i  (2e_i-1)  - (e_i-1) z_i d_i)\\
= & \frac{1}{n_i} \cdot (e_i (d_iz_i + w_i  (2e_i -1))\\
= & \frac{1}{n_i} \cdot (-n_ie_i)\\
= & -e_i\\
\begin{pmatrix}
  x_i & y_i 
 \end{pmatrix} \cdot 
 \begin{pmatrix}
 d_i & e_i \\
 e_i-1 & 0
 \end{pmatrix} \cdot
 \begin{pmatrix}
 x_i \\
 x_i
 \end{pmatrix}
 =& \frac{1}{{n_i}^2}\cdot \begin{pmatrix}
  2e_i-1 & -d_i 
 \end{pmatrix} \cdot 
 \begin{pmatrix}
 d_i(2e_i-1)-e_id_i\\
 (e_i-1)(2e_i-1)
 \end{pmatrix}\\
= & \frac{1}{{n_i}^2} \cdot (d_i  (2e_i-1)^2 - e_i  (2e_i -1) d_i - d_i(e_i-1)(2e_i-1))\\
= & \frac{d}{n_i} \cdot (4e_i^2-4e_i+1-2e_i^2+e_i-2e_i^2+3e_i-1)\\
= & \frac{1}{n_i} \cdot (0)\\
= & 0.
\end{split}
\end{equation*}
Therefore, for $i=1,2,3$, $\widetilde{H_i}:=\spn(x_ia_i+y_ib_i)$ is also a metabolizer for $K_i$ and we have a Seifert matrix $\widetilde{M_i} =\begin{pmatrix}
  * & -e_i +1 \\ 
  -e_i & 0 
 \end{pmatrix}$ for $K_i$ with respect to a symplectic basis $\{z_ia_i+w_ib_i, x_ia_i+y_ib_i\}$. Hence we can change basis so that the Seifert matrix takes value of either $\begin{pmatrix}
  * & e_i \\ 
  e_i-1 & 0 
 \end{pmatrix}$ or $\begin{pmatrix}
  * & -e_i +1 \\ 
  -e_i & 0 
 \end{pmatrix}$ for $i=1,2,3$. Therefore we can assume $|e_i| > |e_i-1|$ for each $i=1,2,3$, which implies that $|e_1|\cdot|e_2|\cdot|e_3| > |e_1-1|\cdot|e_2-1|\cdot|e_3-1|$.
 
 Now, let $F=F_1\natural F_2 \natural F_3$ be a Seifert surface for $K$ and it has Seifert matrix $$M =
 \begin{pmatrix}
  * & e_1 & 0 & 0 & 0 & 0 \\
  e_1-1 & 0 & 0 & 0 & 0 & 0 \\
  0 & 0 & * & e_2 & 0 & 0 \\
  0 & 0 & e_2-1 & 0 & 0 & 0 \\
  0 & 0 & 0 & 0 & * & e_3 \\
  0 & 0 & 0 & 0 & e_3-1 & 0
 \end{pmatrix}.$$ Then by Theorem $4.3$ we have $S_{K,H} \supseteq ((e_1-1)(e_2-1)(e_3-1)-e_1e_2e_3)\cdot\mathbb{Z}$ where $((e_1-1)(e_2-1)(e_3-1)-e_1e_2e_3)\neq 0$ as desired.
\end{proof}

We will end this paper with few remarks.

\begin{remark}$ $
\begin{enumerate}
\item Notice that Corollary $4.5$ is a special case of Corollary $4.6$, where $K_1,K_2,K_3$ are unknots.
\item For the proof of corollary $4.6$ we only needed the assumption that $K$ has the same Seifert matrix as the connected sum of three genus one algebraically slice knots.
\end{enumerate}
\end{remark}


\bibliographystyle{alpha}
\bibliography{knotbib}

\begin{figure}[p]
\centering
\includegraphics[width=5.5in]{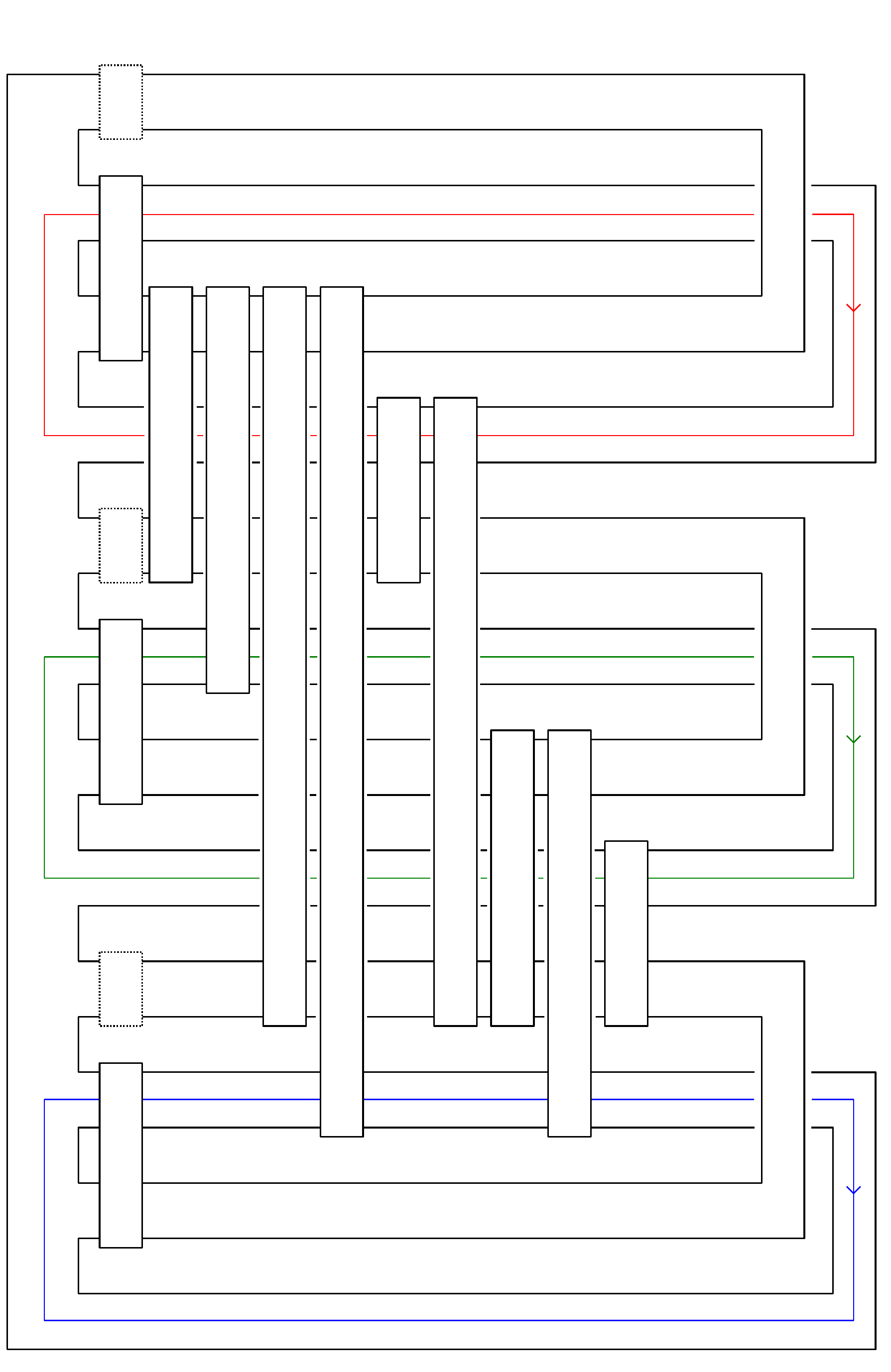}
\put(-4.8,7.95){\rotatebox{-90}{$m_{11}$}}
\put(-4.8,7){\rotatebox{-90}{$a-1$}}
\put(-4.8,5.23){\rotatebox{-90}{$m_{33}$}}
\put(-4.8,4.25){\rotatebox{-90}{$b-1$}}
\put(-4.8,2.48){\rotatebox{-90}{$m_{55}$}}
\put(-4.8,1.50){\rotatebox{-90}{$c-1$}}
\put(-4.5,6.4){\rotatebox{-90}{$m_{13}$}}
\put(-4.15,6.4){\rotatebox{-90}{$x_1$}}
\put(-3.8,6.4){\rotatebox{-90}{$m_{15}$}}
\put(-3.45,6.4){\rotatebox{-90}{$y_1$}}
\put(-3.1,5.6){\rotatebox{-90}{$x_2$}}
\put(-2.75,5.6){\rotatebox{-90}{$y_2$}}
\put(-2.4,3.7){\rotatebox{-90}{$m_{35}$}}
\put(-2.05,3.7){\rotatebox{-90}{$z_1$}}
\put(-1.68,2.8){\rotatebox{-90}{$z_2$}}
\put(-5.4,6.5){\rotatebox{-90}{$\gamma_1$}}
\put(-5.4,3.8){\rotatebox{-90}{$\gamma_2$}}
\put(-5.4,1.1){\rotatebox{-90}{$\gamma_3$}}

\caption{Knot $K$ with the simplest Seifert surface and the derivative $\{\gamma_1, \gamma_2, \gamma_3 \}$ associated with $H$. (Dotted box represents full twists and solid box represents full twists between two bands with no twist on each bands.)}
\end{figure}

\begin{figure}[p]
\centering
\includegraphics[width=5.5in]{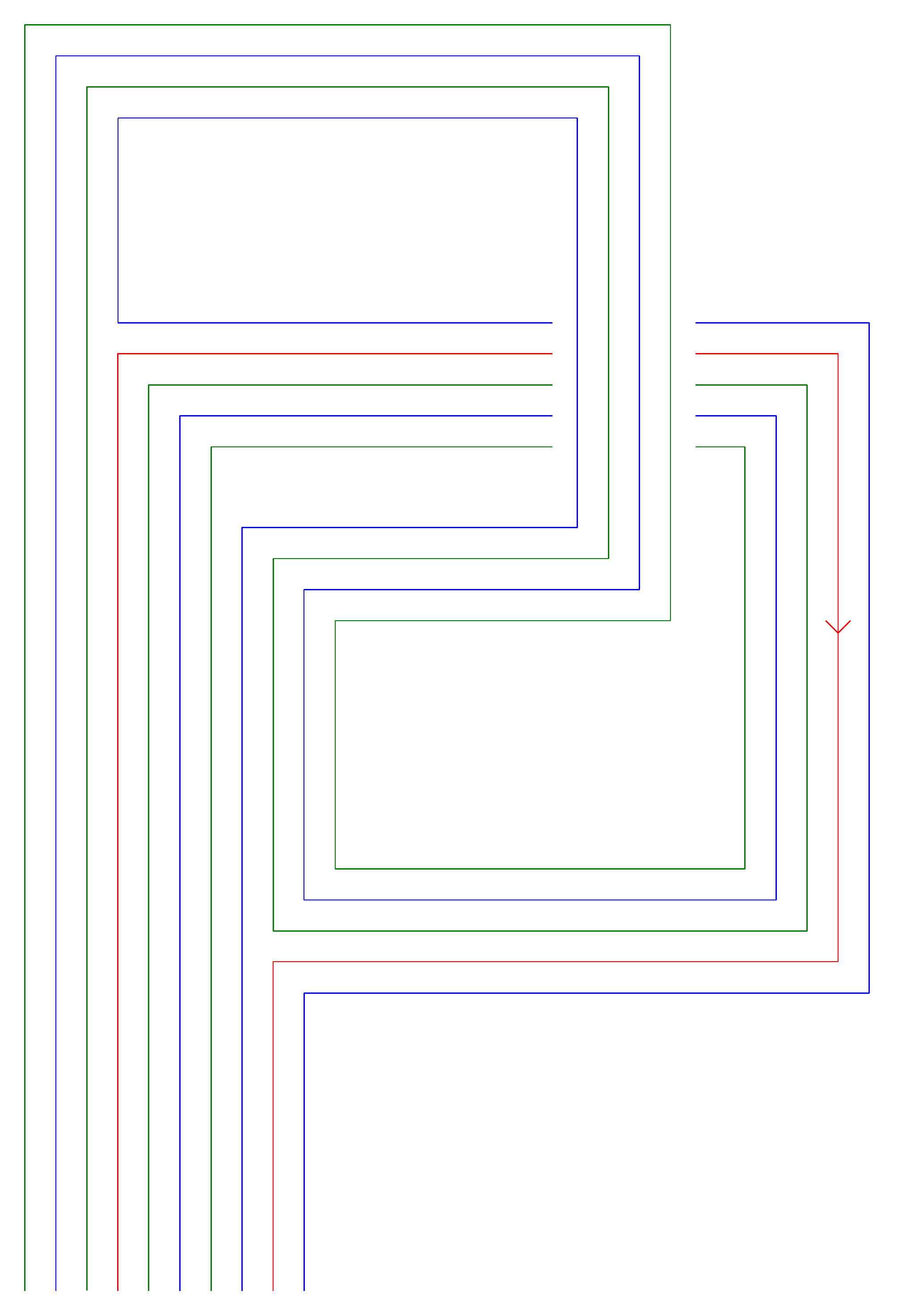}
\put(-0.47,4.5){\rotatebox{-90}{$\gamma_{1,1}$}}
\caption{Link $L_1 = \{ \gamma_{1,1}, \gamma_{1,2}, \gamma_{1,3} \}$ embedded in $\widetilde{\Sigma_1}$. (Obvious $\widetilde{\Sigma_1}$ is omitted for simplicity.)}
\end{figure}

\begin{figure}[p]
\centering
\includegraphics[width=5.5in]{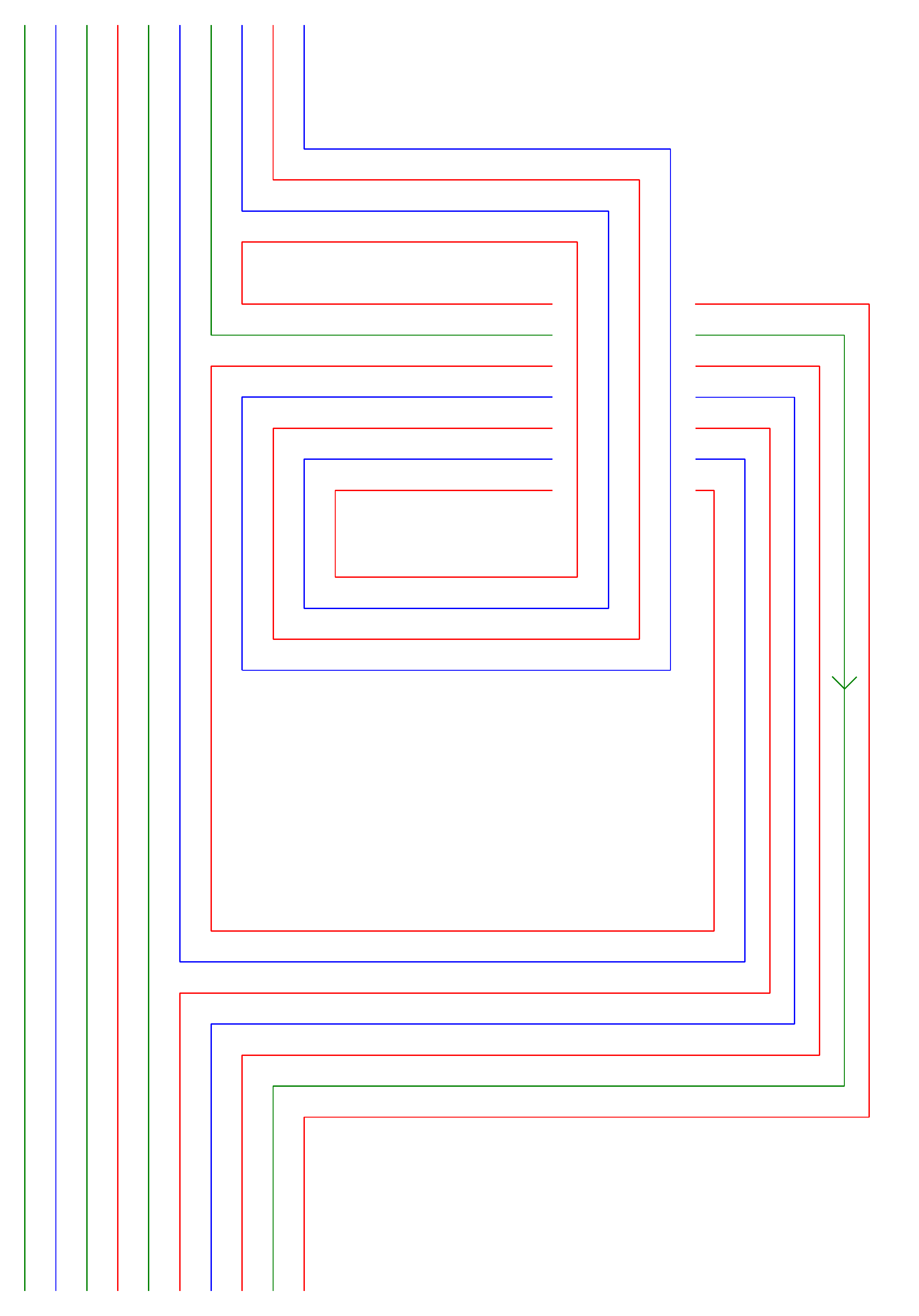}
\put(-0.44,4.1){\rotatebox{-90}{$\gamma_{1,2}$}}
\caption{Link $L_1 = \{ \gamma_{1,1}, \gamma_{1,2}, \gamma_{1,3} \}$ embedded in $\widetilde{\Sigma_1}$. (Obvious $\widetilde{\Sigma_1}$ is omitted for simplicity.)}
\end{figure}

\begin{figure}[p]
\centering
\includegraphics[width=5.5in]{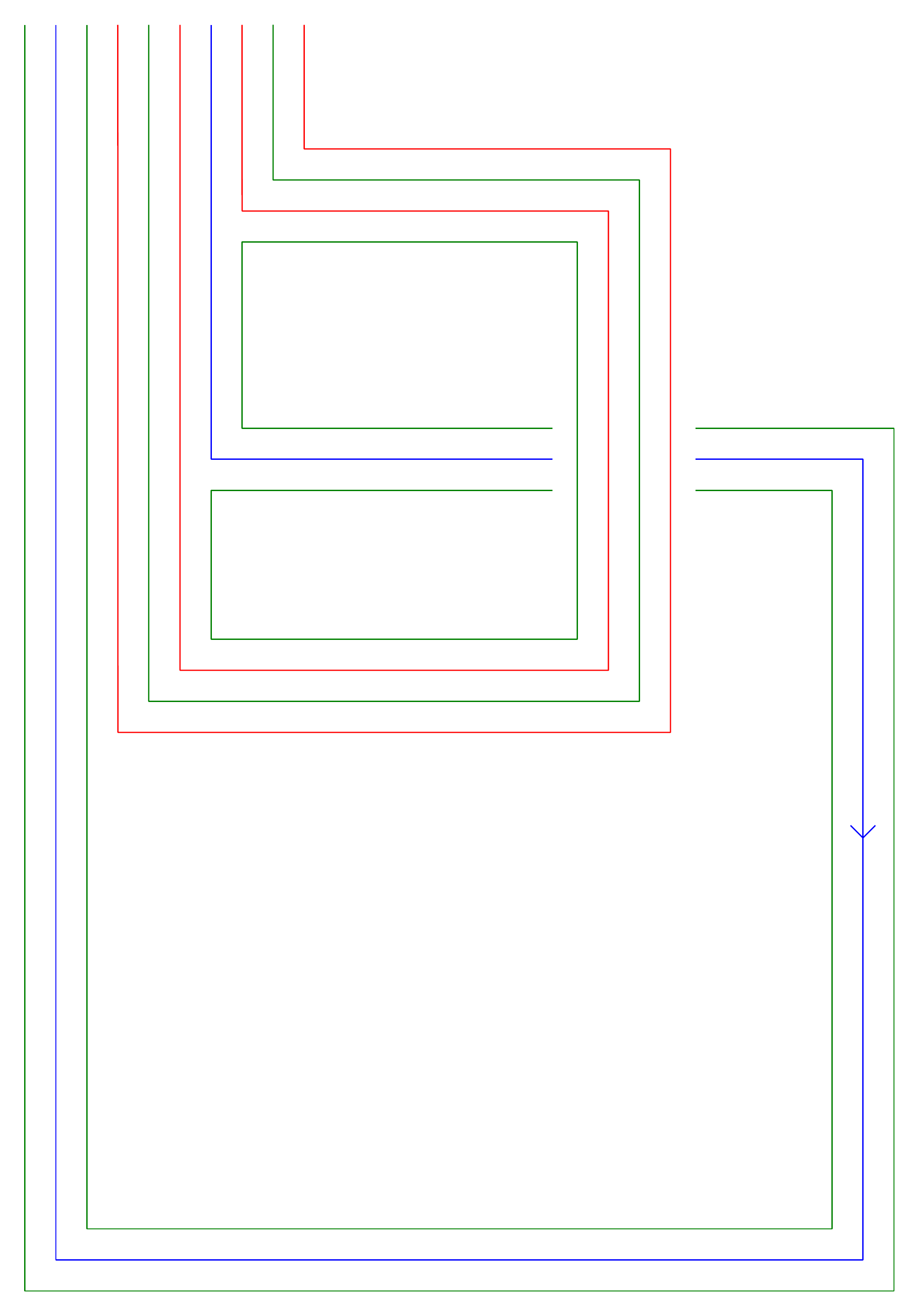}
\put(-0.32,3.3){\rotatebox{-90}{$\gamma_{1,3}$}}
\caption{Link $L_1 = \{ \gamma_{1,1}, \gamma_{1,2}, \gamma_{1,3} \}$ embedded in $\widetilde{\Sigma_1}$. (Obvious $\widetilde{\Sigma_1}$ is omitted for simplicity.)}
\end{figure}

\begin{figure}[p]
\centering
\includegraphics[width=5.5in]{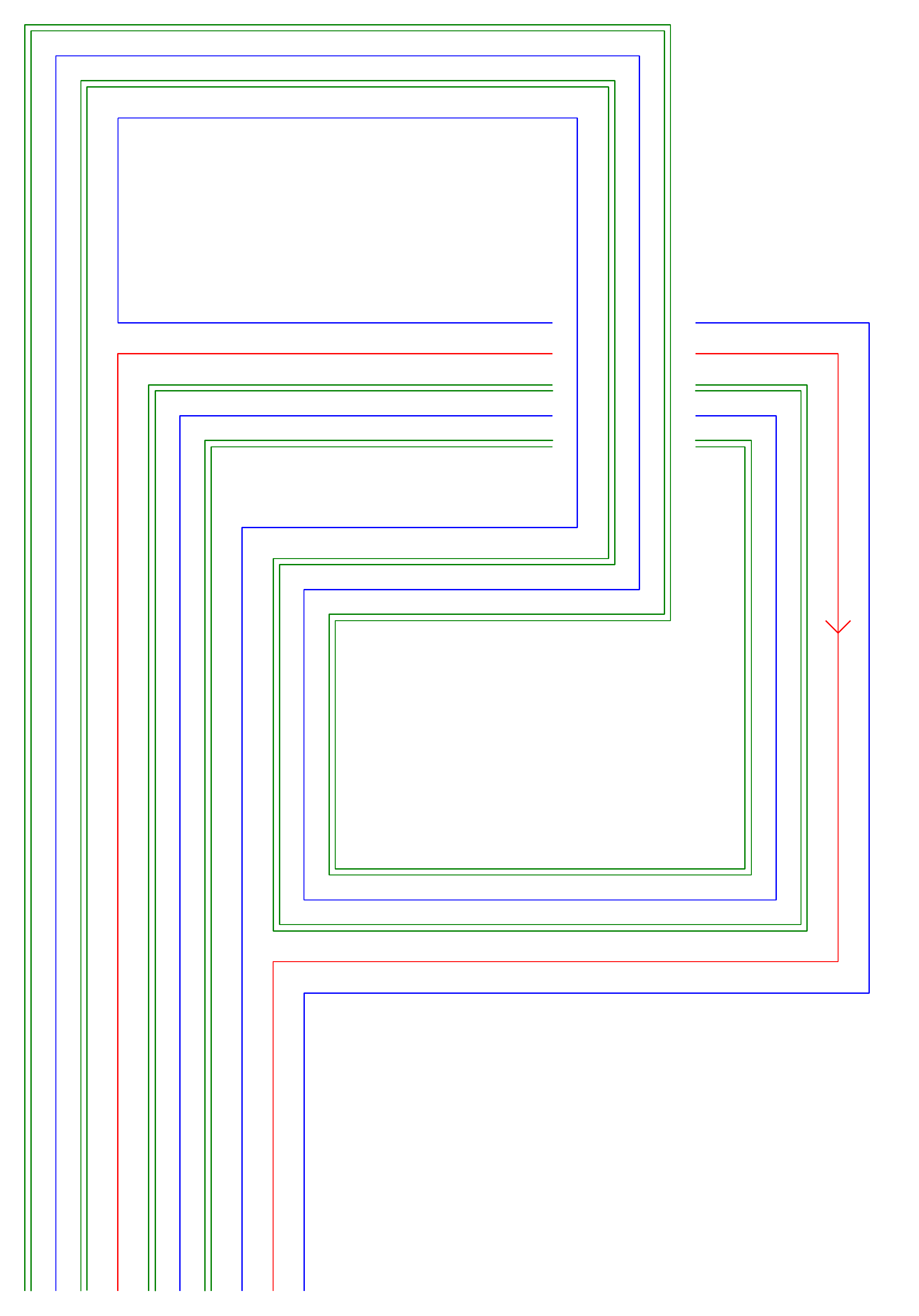}
\put(-0.47,4.5){\rotatebox{-90}{$\gamma_{2,1}$}}
\caption{Link $L_2 = \{ \gamma_{2,1}, \gamma_{2,2}, \gamma_{2,3} \}$ embedded in $\widetilde{\Sigma_2}$. (Obvious $\widetilde{\Sigma_2}$ is omitted for simplicity.)}
\end{figure}

\begin{figure}[p]
\centering
\includegraphics[width=5.5in]{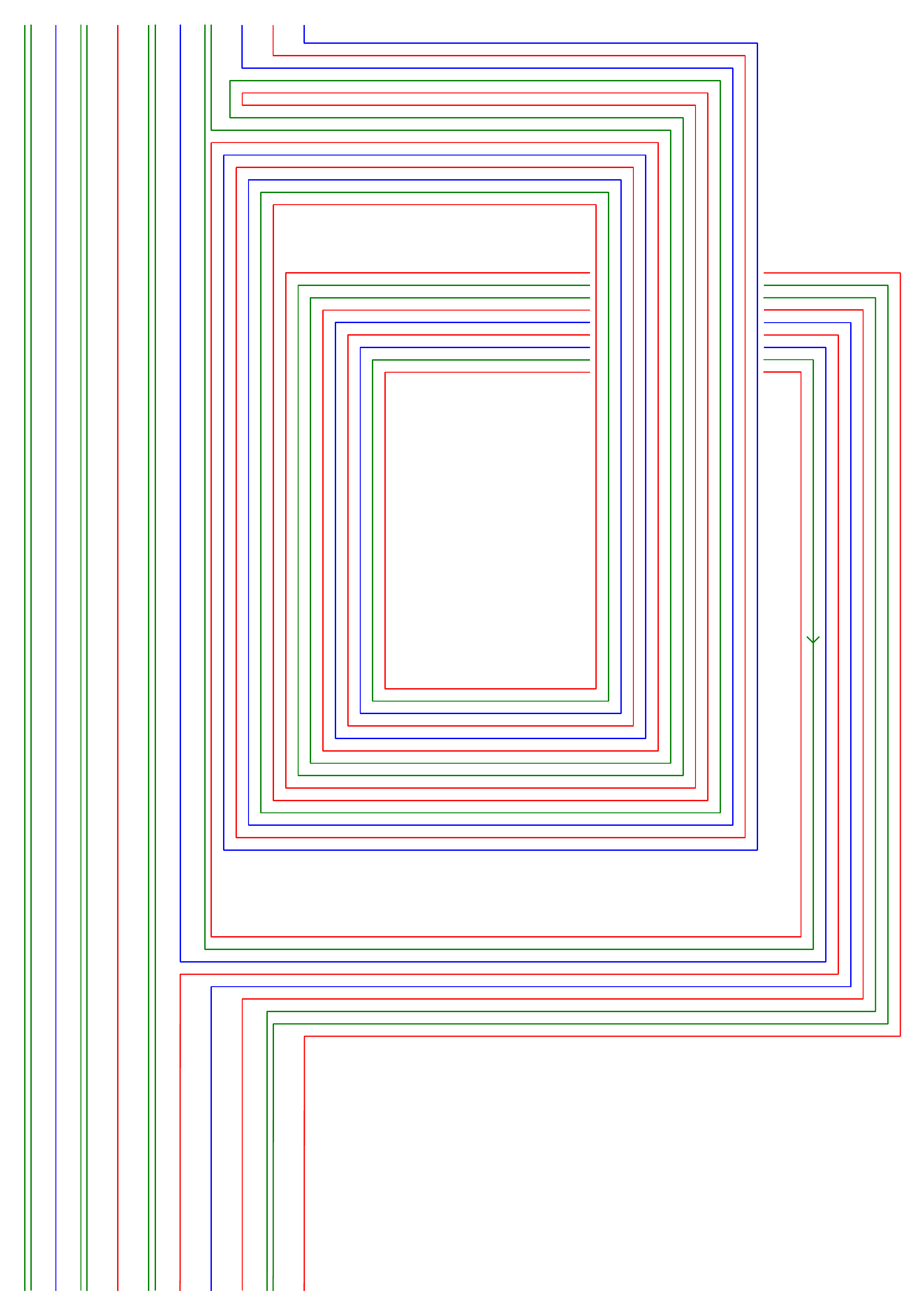}
\put(-0.85,4.2){\rotatebox{-90}{$\gamma_{2,2}$}}
\caption{Link $L_2 = \{ \gamma_{2,1}, \gamma_{2,2}, \gamma_{2,3} \}$ embedded in $\widetilde{\Sigma_2}$. (Obvious $\widetilde{\Sigma_2}$ is omitted for simplicity.)}
\end{figure}

\begin{figure}[p]
\centering
\includegraphics[width=5.5in]{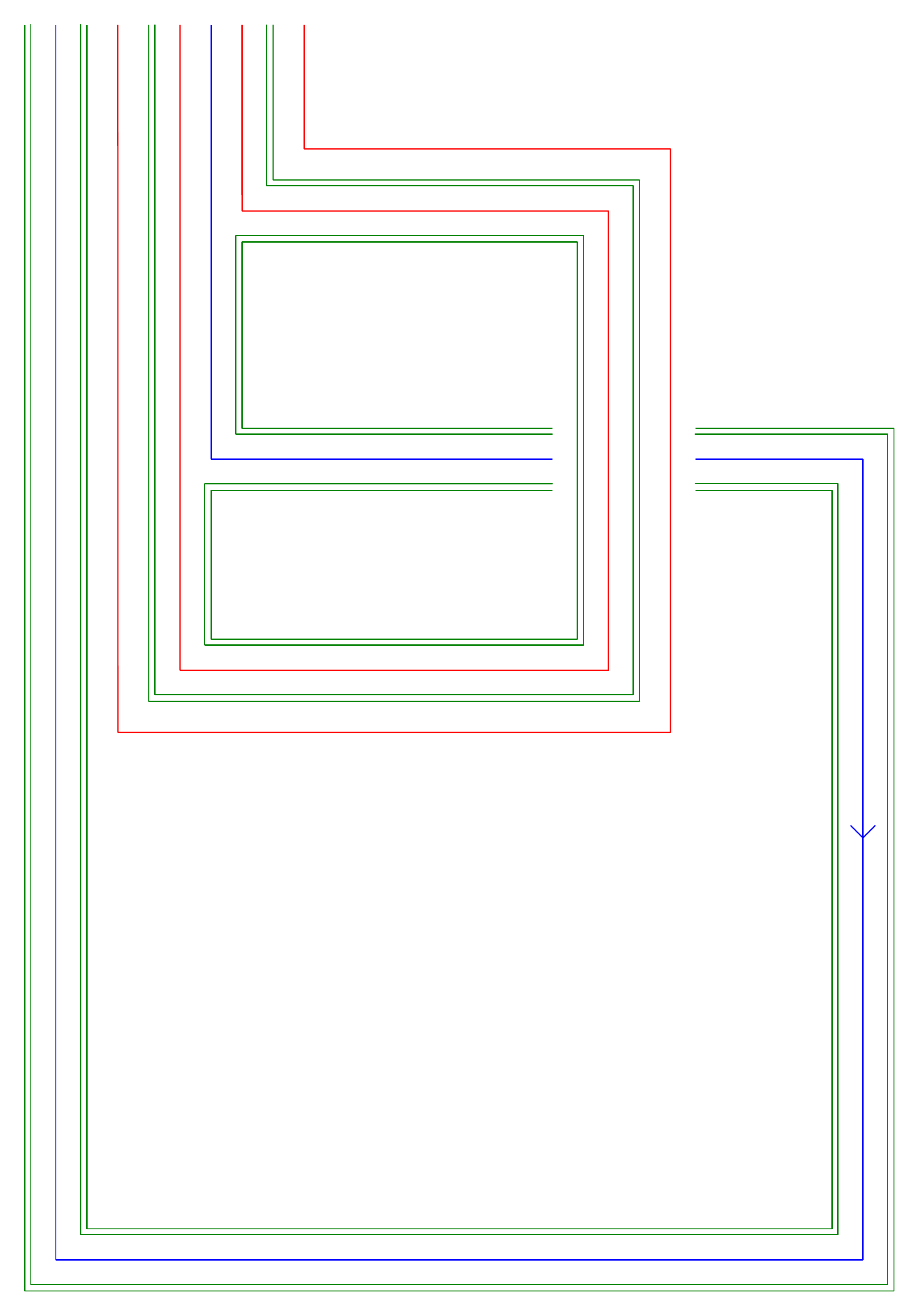}
\put(-0.32,3.3){\rotatebox{-90}{$\gamma_{2,3}$}}
\caption{Link $L_2 = \{ \gamma_{2,1}, \gamma_{2,2}, \gamma_{2,3} \}$ embedded in $\widetilde{\Sigma_2}$. (Obvious $\widetilde{\Sigma_2}$ is omitted for simplicity.)}
\end{figure}

\end{document}